%% file: balcerzak_kubarski_v3.tex
\documentclass[a4paper,12pt]{article}
\usepackage{amsfonts}
\usepackage{amsmath}
\usepackage{graphicx}
\usepackage{amscd}
\usepackage{amssymb}
\usepackage{hyperref}
\usepackage{appendix}%
\providecommand{\U}[1]{\protect\rule{.1in}{.1in}}

\newtheorem{theorem}{Theorem}[section]

\newtheorem{corollary}[theorem]{Corollary}
\newtheorem{definition}[theorem]{Definition}
\newtheorem{example}[theorem]{Example}

\newtheorem{problem}[theorem]{Problem}
\newtheorem{remark}[theorem]{Remark}

\numberwithin{equation}{section}
\newenvironment{proof}[1][Proof]{\textbf{#1.} }{\ \rule{0.5em}{0.5em}}

\input{diagram}
\begin{document}

\title{Properties of $\ $the Exotic Characteristic Homomorphism for a Pair of Lie
Algebroids, Relationship with the Koszul Homomorphism for\\a Pair of Lie algebras}
\date{by\\
\ \\
Bogdan Balcerzak and Jan Kubarski}
\maketitle

\begin{abstract}
We examine functorial and homotopy properties of the exotic characteristic
homomorphism in the category of Lie algebroids which was lastly obtained by
the authors in \cite{B-K}. This homomorphism depends on a triple $\left(
A,B,\nabla\right)  $ where $B\subset A$ are regular Lie algebroids, both over
the same regular foliated manifold $\left(  M,F\right)  $, and $\nabla$ is a
flat $L$-connection in $A$, where $L$ is an arbitrary Lie algebroid over $M$.
The Rigidity Theorem (i.e. the independence from the choice of homotopic Lie
subalgebroids of $B$) is obtained. The exotic characteristic homomorphism is
factorized by one (called universal) obtained for a pair of regular Lie
algebroids. We raise the issue of injectivity of the universal homomorphism
and establish injectivity for special cases. Here the Koszul homomorphism for
pairs of isotropy Lie algebras plays a major role.

\end{abstract}

\section{Introduction}

\footnotetext{\hspace{-0.65cm}\textit{Mathematics Subject Classification
(2010):} 57R20 (primary); \ 58H05, 17B56, 53C05 (secondary)
\par
\hspace{-0.37cm}\textit{Keywords:} secondary (exotic) flat characteristic
classes, secondary characteristic homomorphism, Lie algebroid, functoriality,
homotopy invariance}

In \cite{B-K} we constructed some secondary (exotic)\ characteristic
homomorphism
\[
\Delta_{\left(  A,B,\nabla\right)  \#}:\mathsf{H}^{\bullet}\left(  {{\pmb{g}}%
},B\right)  \longrightarrow\mathsf{H}^{\bullet}\left(  L\right)
\]
for a triple $\left(  A,B,\nabla\right)  $, in which $A$ is a regular Lie
algebroid over a foliated manifold $\left(  M,F\right)  $, $B$ its regular
subalgebroid on the same foliated manifold $\left(  M,F\right)  $,
${{\pmb{g}}}$ the kernel of the anchor of $A$, and $\nabla:L\rightarrow A$ a
flat $L$-connection in $A$ for an arbitrary Lie algebroid $L$ over $M$. The
domain of this homomorphism is the Lie algebroid analog to the relative
cohomology algebra for a pair of Lie algebras defined in \cite{Chev-Eil}.
$\Delta_{\left(  A,B,\nabla\right)  \#}$ generalizes some known secondary
characteristic classes: for flat principal fibre bundles with a reduction
(Kamber, Tondeur \cite{K-T3}) and two approaches to flat characteristic
classes for Lie algebroids, the one for regular Lie algebroids due to Kubarski
\cite{K5} and the one for representations of not necessarily regular Lie
algebroids on vector bundles developed by Crainic (\cite{Cr}, \cite{Cr-F}).

For $L=A$ and $\nabla=\operatorname*{id}_{A}$ we obtain a new universal
characteristic homomorphism $\Delta_{\left(  A,B\right)  \#}$, which
factorizes the characteristic homomorphism $\Delta_{\left(  A,B,\nabla\right)
\#}$ for each flat $L$-connection $\nabla:L\rightarrow A$, i.e.%
\begin{equation}
\Delta_{\left(  A,B,\nabla\right)  \#}=\nabla^{\#}\circ\Delta_{\left(
A,B\right)  \#}.
\end{equation}
Clearly, no class from the kernel of $\Delta_{\left(  A,B\right)  \#}$ is an
obstruction to the fact that the given flat connection $\nabla:L\rightarrow A$
is induced by a connection in $B$. By this reason, we put the following (new)
question for secondary characteristic classes: \textbf{Is the exotic universal
characteristic homomorphism }$\Delta_{\left(  A,B\right)  \#}$\textbf{ a
monomorphism? }In Section \ref{universal}, we give the positive answer under
some assumptions.

We remark that the characteristic homomorphism $\Delta_{\left(  A,B\right)
\#}:\mathsf{H}^{\bullet}\left(  {{\pmb{g}}},B\right)  \longrightarrow
\mathsf{H}^{\bullet}\left(  A\right)  $ really depends only on the inclusion
$i:B\hookrightarrow A$, \cite{B-K}. For this inclusion --- as for any
homomorphism of Lie algebroids --- I.~Vaisman in \cite{Vaisman} defined
secondary characteristic classes $\mu_{2h-1}\left(  i\right)  $ lying in
$\mathsf{H}^{4h-3}\left(  B\right)  $, i.e. in a different group of cohomology
than universal characteristic classes (which belong to $\mathsf{H}^{\bullet
}\left(  A\right)  $). The detailed relationships between these frameworks for
secondary characteristic classes will be the subject of the next paper. We
point out only (see \cite{Fernandes}, \cite{Cr}, \cite{Cr-F}, \cite{B-K}) that
the modular class $\operatorname{mod}(\widetilde{A})$ of a Lie algebroid
$\widetilde{A}$ (for a definition see \cite{Weinstein},
\cite{Evens-Lu-Weinstein}) --- which is equal to the first secondary
characteristic class of the anchor of $\widetilde{A}$ --- in case where the
basic connection $\nabla=(\hat{\nabla},\check{\nabla})$ given in
\cite{Fernandes} is flat, can express in the term of secondary characteristic
homomorphism $\Delta_{\left(  A,B,\nabla\right)  \#}$ for the triple $\left(
A,B,\nabla\right)  $ where $A=\mathcal{A}(\widetilde{A}\oplus T^{\ast}M)$ is a
Lie algebroid of the vector bundle $\widetilde{A}\oplus T^{\ast}M$,
$B=\mathcal{A}(\widetilde{A}\oplus T^{\ast}M,\left\{  h\right\}  )$ is its Lie
subalgebroid being the Lie algebroid of the Riemannian reduction
$(\widetilde{A}\oplus T^{\ast}M,\left\{  h\right\}  )$ and $\nabla
=(\hat{\nabla},\check{\nabla})$; namely $\operatorname{mod}(\widetilde{A})$
and $\Delta_{\left(  A,B,\nabla\right)  \#}\left(  y_{1}\right)  $ are equal
up to a constant. The first secondary characteristic class $\mu_{1}\left(
i\right)  $ of the considered inclusion $i:B\hookrightarrow A$ equals
$\operatorname{mod}\left(  B\right)  -i^{\#}\left(  \operatorname{mod}%
A\right)  $. So, it can be expressed then in terms of characteristic classes
from the images of suitable characteristic homomorphisms of the form
$\Delta_{\left(  A^{\prime},B^{\prime},\nabla^{\prime}\right)  \#}$
constructed in \cite{B-K}.

The meaning of the classical exotic characteristic homomorphism for a
principal bundle with a given reduction consist in that it measures the
incompatibility of two geometric structures on a given principal bundle: its
reduction and a flat connection. Namely, if a flat connection is a connection
in a given reduction, this exotic characteristic homomorphism is trivial (i.e.
it is the zero homomorphism in all positive degrees). The exotic
characteristic homomorphism $\Delta_{\left(  A,B,\nabla\right)  \#}$ for Lie
algebroids has the similar meaning.

The classical exotic homomorphism for given principal bundle $P$ and its
reduction $P^{\prime}$ has stronger property: it is trivial if a given flat
connection has values in any reduction homotopic to $P^{\prime}$ (in some
cases every two $H$-reduction are homotopic \cite{K-T3}).

Chapter \ref{functorial} concerns investigation of homotopic properties of the
generalized exotic homomorphism $\Delta_{\left(  A,B,\nabla\right)  \#}$ for
Lie algebroids. We examine here the notion of homotopic Lie subalgebroids,
which was introduced in \cite{K5} as a natural generalization of the notion of
homotopic two $H$-reductions of a principal bundle. We show also functorial
properties of the considered homomorphism $\Delta_{\left(  A,B,\nabla\right)
\#}$ with respect to homomorphisms of Lie algebroids (not necessary over
identity on the base manifold).

Chapter \ref{universal} concerns exotic universal characteristic homomorphism
$\Delta_{\left(  A,B\right)  \#}$ in two cases. First, for the trivial case of
Lie algebroids over a point, i.e. for Lie algebras. This universal
homomorphism is, in fact, equivalent (up to the sign) to the well known
\textquotedblright Koszul homomorphism\textquotedblright\ $\Delta_{\left(
\mathfrak{g},\mathfrak{h}\right)  \#}:\mathsf{H}^{\bullet}\left(
{\mathfrak{g}}/{\mathfrak{h}}\right)  \rightarrow\mathsf{H}^{\bullet}\left(
\mathfrak{g}\right)  $ for a pair of Lie algebras $\left(  \mathfrak{g}%
,\mathfrak{h}\right)  $, $\mathfrak{h}\subset\mathfrak{g}$ \cite{K},
\cite{GHV}. In \cite{GHV} the injectivity of $\Delta_{\left(  \mathfrak{g}%
,\mathfrak{h}\right)  \#}$ is considered and used to investigate of the
cohomology algebra of the homogeneous manifolds $G/H$. Next, applying the Lie
functor for principal fibre bundles it gives a new universal\ homomorphism for
the reduction of a principal fibre bundle. It factorizes the standard
secondary characteristic homomorphisms $\Delta_{\left(  P,P^{\prime}%
,\omega\right)  \#}$ for any flat connections $\omega$ in $P$. In Section
\ref{ostatni}, using functorial properties of the inclusion $\iota_{x}:\left(
{{\pmb{g}}}_{x},{{\pmb{h}}}_{x}\right)  \rightarrow\left(  A,B\right)  $ over
the map $\left\{  x\right\}  \hookrightarrow M$, where ${{\pmb{g}}}%
_{x},{{\pmb{h}}}_{x}$ are isotropy algebras of Lie algebroids $A$ and $B$ at
$x\in M$, we show connection of the exotic universal characteristic
homomorphism $\Delta_{\left(  A,B\right)  \#}$ with the Koszul homomorphism
for isotropy Lie algebras $\left(  {{\pmb{g}}}_{x},{{\pmb{h}}}_{x}\right)  $.
We find some conditions under which the characteristic homomorphism
$\Delta_{\left(  A,B\right)  \#}$, for a pair $B\subset A$, is a monomorphism.
The presented considerations show that the Koszul homomorphism plays an
essential role for the study of exotic characteristic classes.

In the paper we suppose that the reader is familiar with Lie algebroids and
for more about Lie algebroids and its connections we refer to \cite{M1},
\cite{Higgins-Mackenzie}, \cite{K6}, \cite{B-K-W}, \cite{Fernandes}.

\section{Construction of Exotic Characteristic
Homomorphism\label{Unifying_homomorphism}}

We shall briefly explain the construction of the exotic characteristic
homomorphism and the universal exotic characteristic homomorphism on Lie
algebroids from \cite{B-K}.

Let $\left(  {A},[\![\cdot,\cdot]\!],\#_{A}\right)  $ be a regular Lie
algebroid over a foliated manifold $\left(  M,F\right)  $, $B$ its regular
\textbf{subalgebroid} on the same foliated manifold $\left(  M,F\right)  $,
$L$ a Lie algebroid over $M$ and $\nabla:L\rightarrow A$ a \textbf{flat}
$L$-connection in $A$. We call the triple%
\[
\left(  A,B,\nabla\right)
\]
an FS-\emph{Lie algebroid}. Let $\lambda:F\rightarrow B$ be an arbitrary
connection in $B$. Then $j\circ\lambda:F\rightarrow A$ is a connection in $A$.
Let $\breve{\lambda}:A\rightarrow{{\pmb{g}}}$ be its connection form.
Summarizing, we have a flat $L$-connection $\nabla:L\rightarrow A$ in $A$ and
the following commutative diagram%
\[
\xext=2100\yext=700\begin{picture}(\xext,\yext)(\xoff,\yoff) \putmorphism(0,620)(1,0)[0`\pmb{g}`]{500}1a \putmorphism(550,620)(1,0)[`A`]{450}1a \putmorphism(1050,620)(1,0)[``\#_{A}]{400}1a \putmorphism(1050,70)(-1,0)[``\lambda]{400}{-1}b \putmorphism(1500,620)(1,0)[F`0`]{500}1a \putmorphism(0,100)(1,0)[0`\pmb{h}`]{500}1a \put(492,170){{$\cup$}}\putmorphism(550,100)(1,0)[`B`]{450}1a \putmorphism(1050,100)(1,0)[`F`\#_{B}]{450}1a \putmorphism(1550,100)(1,0)[`0.`]{450}1b \put(550,624){{$\subset$}} \putmorphism(500,600)(0,1)[``]{460}{-1}r \put(992,170){{$\cup$}}\putmorphism(1000,600)(0,1)[``]{460}{-1}r \put(935,590){\vector(-1,0){370}}\put(750,540){\makebox(0,0){$\breve{\lambda} $}} \put(750,680){\makebox(0,0){$\iota$}} \put(1050,380){\makebox(0,0){$j$}} \put(1490,170){\line(0,1){370}} \put(1515,170){\line(0,1){370}} \end{picture}
\]
The homomorphism $\omega_{B,\nabla}:L\longrightarrow{{\pmb{g}/\pmb{h}}}$,
$\omega_{B,\nabla}\left(  w\right)  =[-(\breve{\lambda}\circ\nabla)\left(
w\right)  ]$ does not depend on the choice of an auxiliary connection
$\lambda:F\rightarrow A$ and $\omega_{B,\nabla}=0$ if $\nabla$ takes values in
$B$. Let us define a homomorphism of algebras
\begin{equation}
\Delta_{\left(  A,B,\nabla\right)  }:\Gamma{\LARGE (}\bigwedge\nolimits^{k}%
\left(  {{\pmb{g}/\pmb{h}}}\right)  ^{\ast}{\LARGE )}\longrightarrow
\Omega\left(  L\right)  ,
\end{equation}%
\[
{\large (}\Delta_{\left(  A,B,\nabla\right)  }\Psi{\large )}_{x}\left(
w_{1}\wedge\ldots\wedge w_{k}\right)  =\left\langle \Psi_{x},\omega_{B,\nabla
}\left(  w_{1}\right)  \wedge\ldots\wedge\omega_{B,\nabla}\left(
w_{k}\right)  \right\rangle ,\ \ w_{i}\in L_{|x}.
\]
In the algebra $\Gamma\left(  \bigwedge\left(  {{\pmb{g}/\pmb{h}}}\right)
^{\ast}\right)  $ we distinguish the subalgebra \label{sec3.1 (A) copy(1)}%
$\left(  \Gamma\left(  \bigwedge\left(  {{\pmb{g}/\pmb{h}}}\right)  ^{\ast
}\right)  \right)  ^{\Gamma\left(  B\right)  }$ of invariant cross-sections
with respect to the representation of the Lie algebroid $B$ in the vector
bundle $\bigwedge\left(  {{\pmb{g}/\pmb{h}}}\right)  ^{\ast}$, associated to
the adjoint one $\operatorname*{ad}\nolimits_{B,{{\pmb{h}}}}:B\rightarrow
\operatorname*{A}\left(  {{\pmb{g}/\pmb{h}}}\right)  $, $\operatorname*{ad}%
\nolimits_{B,{{\pmb{h}}}}\left(  \xi\right)  \left(  \left[  \nu\right]
\right)  =\left[  [\![\xi,\nu]\!]\right]  $, $\xi\in\Gamma\left(  B\right)  $,
$\nu\in\Gamma\left(  {{\pmb{g}}}\right)  $. Recall, that $\Psi\in\left(
\Gamma\left(  \bigwedge^{k}\left(  {{\pmb{g}/\pmb{h}}}\right)  ^{\ast}\right)
\right)  ^{\Gamma\left(  B\right)  }$ if and only if%
\[
\left(  \#_{B}\circ\xi\right)  \hspace{-0.1cm}\left\langle \Psi,\left[
\nu_{1}\right]  \wedge\hspace{-0.05cm}\ldots\hspace{-0.05cm}\wedge\left[
\nu_{k}\right]  \right\rangle \hspace{-0.05cm}=\hspace{-0.05cm}\sum_{j=1}%
^{k}\left(  -1\right)  ^{j-1}\hspace{-0.1cm}\left\langle \Psi,\left[
[\![j\circ\xi,\nu_{j}]\!]\right]  \wedge\left[  \nu_{1}\right]  \wedge
\hspace{-0.05cm}\ldots\hat{\jmath}\ldots\hspace{-0.05cm}\wedge\left[  \nu
_{k}\right]  \right\rangle
\]
for all $\xi\in\Gamma\left(  B\right)  $ and $\nu_{j}\in\Gamma\left(
{{\pmb{g}}}\right)  $ (see \cite{K1}). In the space $\left(  \Gamma\left(
\bigwedge\left(  {{\pmb{g}/\pmb{h}}}\right)  ^{\ast}\right)  \right)
^{\Gamma\left(  B\right)  }$ of invariant cross-sections there exists a
differential operator $\bar{\delta}$ defined by%
\[
\hspace{-0.1cm}\left\langle \bar{\delta}\Psi,\left[  \nu_{1}\right]
\wedge\ldots\wedge\left[  \nu_{k}\right]  \right\rangle =%
{\displaystyle\sum\limits_{i<j}}
\left(  -1\right)  ^{i+j+1}\hspace{-0.1cm}\left\langle \Psi,\left[
[\![\nu_{i},\nu_{j}]\!]\right]  \wedge\left[  \nu_{1}\right]  \wedge\ldots
\hat{\imath}\ldots\hat{\jmath}\ldots\wedge\left[  \nu_{k}\right]
\right\rangle ,
\]
(see \cite{K5}) and we obtain the cohomology algebra
\[
\mathsf{H}^{\bullet}\left(  {{\pmb{g}}},B\right)  :=\mathsf{H}^{\bullet
}{\LARGE ((}\Gamma{\LARGE (}\bigwedge\left(  {{\pmb{g}/\pmb{h}}}\right)
^{\ast}{\Large ))}^{\Gamma\left(  B\right)  },\bar{\delta}{\Large )}.
\]
The homomorphism $\Delta_{\left(  A,B,\nabla\right)  }$ commutes with the
differentials $\bar{\delta}$ and $d_{L}$, where $d_{L}$ is the differential
operator in $\Omega\left(  L\right)  =\Gamma\left(  \bigwedge L^{\ast}\right)
$, see \cite{B-K}. In this way we obtain the cohomology homomorphism
\[
\Delta_{\left(  A,B,\nabla\right)  \#}:\mathsf{H}^{\bullet}\left(  {{\pmb{g}}%
},B\right)  \longrightarrow\mathsf{H}^{\bullet}\left(  L\right)  .
\]
In the case where $L=A\ $and $\nabla=\operatorname{id}_{A}:A\rightarrow A$ is
the identity map, we have particular case of a homomorphism for the pair
$\left(  A,B\right)  $:%
\begin{align*}
\Delta_{\left(  A,B\right)  } &  :=\Delta_{\left(  A,B,\operatorname{id}%
_{A}\right)  }:\Gamma{\LARGE (}\bigwedge\nolimits^{k}\left(
{{\pmb{g}/\pmb{h}}}\right)  ^{\ast}{\LARGE )}^{\Gamma\left(  B\right)
}\longrightarrow\Omega\left(  A\right)  ,\\
{\large (}\Delta_{\left(  A,B\right)  }\Psi{\large )}_{x}\left(  \upsilon
_{1}\wedge\ldots\wedge\upsilon_{k}\right)   &  =\langle\Psi_{x},[-\breve
{\lambda}\left(  \upsilon_{1}\right)  ]\wedge\ldots\wedge\lbrack
-\breve{\lambda}\left(  \upsilon_{k}\right)  ]\rangle,\ \ \ \ \upsilon_{i}\in
A_{|x}.
\end{align*}
$\Delta_{\left(  A,B,\nabla\right)  }$ can be written as a composition%
\[%
\begin{CD}
\Delta_{\left( A,B,\nabla\right) }:\operatorname{\Gamma}\left( \bigwedge
\left( {{\pmb{g}/\pmb{h}}}\right) ^{\ast}\right)
@>{   \Delta_{\left( A,B\right) } }>>
\Omega\left( A\right)    @>{\nabla^{\ast} }>> \Omega\left( L\right),
\end{CD}%
\]
where $\nabla^{\ast}$ is the pullback of forms. For this reason,
$\Delta_{\left(  A,B\right)  }$ induces the cohomology homomorphism%
\[
\Delta_{\left(  A,B\right)  \#}:\mathsf{H}^{\bullet}\left(  {{\pmb{g}}%
},B\right)  \longrightarrow\mathsf{H}^{\bullet}\left(  A\right)  ,
\]
which factorizes $\Delta_{\left(  A,B,\nabla\right)  \#}$ for every flat
$L$-connection $\nabla:L\rightarrow A$:%
\begin{equation}%
\begin{CD}
\Delta_{\left( A,B,\nabla\right) \#}:\mathsf{H}^{\bullet}\left( {{\pmb{g},B}%
}\right)      @>{ \Delta_{\left( A,B\right) \#}}>>
\mathsf{H}^{\bullet}\left(A\right)      @>{{\nabla}^{\#} }>> \mathsf
{H}^{\bullet}\left( L\right) .
\end{CD}%
\label{deltacohomology}%
\end{equation}
The map $\Delta_{\left(  A,B,\nabla\right)  \#}$ is called the
\textit{characteristic homomorphism }of the FS-Lie algebroid $\left(
A,B,\nabla\right)  $. We call elements of a subalgebra $\operatorname{Im}%
\Delta_{\left(  A,B,\nabla\right)  \#}\subset\mathsf{H}^{\bullet}\left(
L\right)  $ the \textit{secondary} (\textit{exotic})\textit{\ characteristic
classes} of this algebroid. In particular, $\Delta_{\left(  A,B\right)
\#}=\Delta_{\left(  A,B,\operatorname{id}_{A}\right)  \#}$ is the
characteristic homomorphism of the Lie subalgebroid $B\subset A$, which we
call the \textit{universal exotic characteristic homomorphism}; the
characteristic classes from its image we call the \textit{universal
characteristic classes} of the pair $B\subset A$.

The secondary characteristic homomorphism for FS-Lie algebroids generalizes
the following known characteristic classes: for flat regular Lie algebroids
(Kubarski), for flat principal fibre bundles with a reduction (Kamber,
Tondeur) and for representations of Lie algebroids on vector bundles (Crainic).

\begin{enumerate}
\item For $L=F$\ we obtain the case in which $\nabla:F\rightarrow A$ is a
usual connection in $A$. In this way the exotic characteristic homomorphism is
a generalization of one for a flat regular Lie algebroid given in \cite{K5},
see \cite{B-K}.

\item For $L=TM$\ and $A=TP/G$, and $B=TP^{\prime}/H$ ($P^{\prime}$ is an
$H$-reduction of $P$) we obtain the case equivalent to the standard classical
theory on principal fibre bundles \cite{K-T3} (see \cite{B-K} and Section
\ref{Exotic_for_principal_bundle} below for more details).

\item \cite{B-K} Let $A=\mathcal{A}\left(  \mathfrak{f}\right)  $ be the Lie
algebroid of a vector bundle $\mathfrak{f}$ over a manifold $M$,
$B=\mathcal{A}\left(  \mathfrak{f,}\left\{  h\right\}  \right)  \subset A$ its
Riemannian reduction (\cite{K3}), $L$ a Lie algebroid over $M$, $\nabla
:L\rightarrow\mathcal{A}\left(  \mathfrak{f}\right)  $ an $L$-connection on
$\mathfrak{f}$. Let $\Delta_{\#}$ denote the exotic characteristic
homomorphism for FS-Lie algebroid $\left(  \mathcal{A}\left(  \mathfrak{f}%
\right)  ,\mathcal{A}\left(  \mathfrak{f,}\left\{  h\right\}  \right)
,\nabla\right)  $. If the vector bundle $\mathfrak{f}$ is nonorientable or
orientable and of odd rank $n$, then the domain of $\Delta_{\#}$ is isomorphic
with $\bigwedge\left(  y_{1},y_{3},\ldots,y_{n^{\prime}}\right)  $ where
$n^{\prime}$ is the largest odd integer $\leq n$ and $y_{2k-1}\in
\mathsf{H}^{4k-3}\left(  \operatorname*{End}\mathfrak{f},\mathcal{A}\left(
\mathfrak{f},\left\{  h\right\}  \right)  \right)  $ is represented by the
multilinear trace form $\widetilde{y}_{2k-1}\in\Gamma\left(  \bigwedge
\nolimits^{4k-3}\left(  \operatorname*{End}\mathfrak{f/}\operatorname*{Sk}%
\mathfrak{f}\right)  ^{\ast}\right)  $. Then the image of $\Delta_{\#}$ is
generated by the Crainic classes $u_{1}\left(  \mathfrak{f}\right)  $,
$u_{5}\left(  \mathfrak{f}\right)  $,$\ldots$,$u_{4\left[  \frac{n+3}%
{4}\right]  -3}\left(  \mathfrak{f}\right)  $ (for details about the classes
developed by Crainic see \cite{Cr}, \cite{Cr-F}, \cite{B-K}). If
$\mathfrak{f}$ is orientable of even rank $n=2m$ with a volume form
$\operatorname{v}$, the domain of $\Delta_{\#}$ is additionally generated by
some class $y_{2m}\in\mathsf{H}^{2m}\left(  \operatorname*{End}\mathfrak{f}%
,\mathcal{A}\left(  \mathfrak{f},\left\{  h,\operatorname{v}\right\}  \right)
\right)  $ represented by a form induced by the Pfaffian and where
$\mathcal{A}\left(  \mathfrak{f},\left\{  h,\operatorname{v}\right\}  \right)
$ is the Lie algebroid of the $SO\left(  n,\mathbb{R}\right)  $-reduction
$\mathcal{L}\left(  \mathfrak{f},\left\{  h,\operatorname{v}\right\}  \right)
$ of the frames bundle $\mathcal{L}\mathfrak{f}$ of $\mathfrak{f}$; see
\cite{B-K}. Then the algebra of exotic characteristic classes for $\left(
\mathcal{A}\left(  \mathfrak{f}\right)  ,\mathcal{A}\left(  \mathfrak{f,}%
\left\{  h,\operatorname{v}\right\}  \right)  ,\nabla\right)  $ is generated
by $u_{1}\left(  \mathfrak{f}\right)  $, $u_{5}\left(  \mathfrak{f}\right)
$,$\ldots$,$u_{4\left[  \frac{n+3}{4}\right]  -3}\left(  \mathfrak{f}\right)
\ $and additionally by $\Delta_{\#}\left(  y_{2m}\right)  $. In \cite{B-K} we
give an example of FS-Lie algebroid where the Pfaffian induces the non-zero
characteristic class.
\end{enumerate}

From (\ref{deltacohomology}) one can see that for a pair of regular Lie
algebroids $\left(  A,B\right)  $, $B\subset A$, both over a foliated manifold
$\left(  M,F\right)  $,\ and for an arbitrary element $\zeta\in\mathsf{H}%
^{\bullet}\left(  {{\pmb{g},B}}\right)  $\ there exists a
(universal)\ cohomology class $\Delta_{\left(  A,B\right)  \,\#}\left(
\zeta\right)  \in\mathsf{H}^{\bullet}\left(  A\right)  $\ such that for any
Lie algebroid $L$\ over $M$\ and a flat $L$-connection $\nabla:L\rightarrow A$
the equality%
\[
\Delta_{\left(  A,B,\nabla\right)  \,\#}\left(  \zeta\right)  =\nabla
^{\#}\left(  \Delta_{\left(  A,B\right)  \,\#}\left(  \zeta\right)  \right)
\]
holds. Therefore, no element from the kernel of $\Delta_{\left(  A,B\right)
\,\#}$ can be used to compare the flat connection $\nabla$ with a reduction
$B\subset A$. Hence it is interesting the following

\begin{problem}
Is the characteristic homomorphism $\Delta_{\left(  A,B\right)  \,\#}%
:\mathsf{H}^{\bullet}\left(  {{\pmb{g},B}}\right)  \rightarrow\mathsf{H}%
^{\bullet}\left(  A\right)  $ a monomorphism for a given $B\subset A$?
\emph{The answer yes holds in some cases, see below.}
\end{problem}

\section{Functoriality and Homotopic Properties\label{functorial}}

\subsection{Functoriality}

Let $\left(  A,B\right)  $ and $\left(  A^{\prime},B^{\prime}\right)  $ be two
pairs of regular Lie algebroids over $\left(  M,F\right)  $ and $\left(
M^{\prime},F^{\prime}\right)  ,$ respectively, where $B\subset A$, $B^{\prime
}\subset A^{\prime},$ and let $H:A^{\prime}\rightarrow A$ be a homomorphism of
Lie algebroids over a mapping $f:\left(  M^{\prime},F^{\prime}\right)
\rightarrow\left(  M,F\right)  $ of foliated manifolds such that $H\left[
B^{\prime}\right]  \subset B$. We write $\left(  H,f\right)  :\left(
A^{\prime},B^{\prime}\right)  \rightarrow\left(  A,B\right)  $. Let
$H^{+\,\#}:\mathsf{H}^{\bullet}\left(  {{\pmb{g},B}}\right)  \rightarrow
\mathsf{H}^{\bullet}\left(  {{\pmb{g}}}^{\prime}{,B}^{\prime}\right)  $ be the
homomorphism of cohomology algebras induced by the pullback $H^{+\,\ast
}:\Gamma\left(  \bigwedge\nolimits^{k}\left(  {{\pmb{g}/\pmb{h}}}\right)
^{\ast}\right)  \rightarrow\Gamma\left(  \bigwedge\nolimits^{k}({{\pmb{g}}%
}^{\prime}{{/\pmb{h}}}^{\prime})^{\ast}\right)  $, see \cite[Proposition
4.2]{K5}.

\begin{theorem}
[The functoriality of $\Delta_{\left(  A,B\right)  \,\#}$]\label{functorof0}%
For a given pair of regular Lie algebroids $\left(  A,B\right)  $, $\left(
A^{\prime},B^{\prime}\right)  $ and a homomorphism $\left(  H,f\right)
:\left(  A^{\prime},B^{\prime}\right)  \rightarrow\left(  A,B\right)  \ $we
have the commutativity of the following diagram
\begin{center}
\setsqparms[1`1`1`1;950`510]
\square[\mathsf{H}^{\bullet}({\pmb{g}},B)`\mathsf{H}^{\bullet}(A)`\mathsf
{H}^{\bullet}({\pmb{g}}^{\prime},B^{\prime})`\mathsf{H}^{\bullet}(A^{\prime
}).;
{\Delta}_{ \left(A,B\right)\#}`H^{+ \#}`H^{\#}`{\Delta}_{(A',B')\#} ]
\end{center}%

\end{theorem}

\begin{proof}
One can see that $H^{+}\circ\breve{\lambda}^{\prime}(u^{\prime})-\breve
{\lambda}(Hu^{\prime})\in{{\pmb {h}\;}}$for all\ \ $u^{\prime}\in A^{\prime}$,
where $\lambda$ and $\lambda^{\prime}$ are auxiliary connections in $B$ and
$B^{\prime},$ respectively. Applying this fact, it is sufficient to check the
commutativity of the diagram on the level of forms. The calculations are left
to the reader.
\end{proof}

\begin{definition}
\emph{Let }$\left(  A^{\prime},B^{\prime},\nabla^{\prime}\right)  $\emph{ and
}$\left(  A,B,\nabla\right)  $\emph{ be two FS-Lie algebroids on foliated
manifolds }$\left(  M^{\prime},F^{\prime}\right)  $\emph{ and }$\left(
M,F\right)  ,$\emph{ respectively, where }$\nabla:L\rightarrow A$\emph{ and
}$\nabla^{\prime}:L^{\prime}\rightarrow A^{\prime}$\emph{ are flat
connections. By a }homomorphism\emph{ }%
\[
H:(A^{\prime},B^{\prime},\nabla^{\prime})\longrightarrow\left(  A,B,\nabla
\right)
\]
over\emph{ }$f:\left(  M^{\prime},F^{\prime}\right)  \rightarrow\left(
M,F\right)  $\emph{ we mean a pair }$\left(  H,h\right)  $\emph{ such that:}

\begin{itemize}
\item $H:A^{\prime}\rightarrow A$\emph{ is a homomorphism of regular Lie
algebroids over }$f$\emph{ and }$H\left[  B^{\prime}\right]  \subset
B$\emph{,}

\item $h:L^{\prime}\rightarrow L$\emph{ is also a homomorphism of Lie
algebroids over }$f$\emph{,}

\item $\nabla\circ h=H\circ\nabla^{\prime}.$
\end{itemize}
\end{definition}

Clearly, $h^{\#}\circ\nabla^{\#}=\nabla^{\prime\,\#}\circ H^{\#}$. So, from
(\ref{deltacohomology}) and Theorem \ref{functorof0} we obtain as a corollary
the following theorem:

\begin{theorem}
[\emph{The functoriality of }$\Delta_{\left(  A,B,\nabla\right)  \,\#}$]\emph{
}The following diagram
\begin{center}
\setsqparms[1`1`1`1;950`510]
\square[\mathsf{H}^{\bullet}({\pmb{g}},B)`\mathsf{H}^{\bullet}(L)`\mathsf
{H}^{\bullet}({\pmb{g}}^{\prime},B^{\prime})`\mathsf{H}^{\bullet}(L^{\prime});
{\Delta}_{(A,B,\nabla)\#}`H^{+ \#}`h^{\#}`{\Delta}_{( A',B',\nabla')\#}]
\end{center}
commutes.\ 
\end{theorem}

\subsection{Homotopy Invariance}

We recall the definition of homotopy between homomorphisms of Lie algebroids.

\begin{definition}
\emph{\cite{K4}} \emph{Let }$H_{0},\,H_{1}:L^{\prime}\rightarrow L$\emph{ be
two homomorphisms of Lie algebroids. By a }homotopy joining $H_{0}$\ to\emph{
}$H_{1}$\emph{ we mean a homomorphism of Lie algebroids }%
\[
H:T\mathbb{R}\times L^{\prime}\longrightarrow L,
\]
\emph{such that }$H\left(  \theta_{0},\cdot\right)  =H_{0}$\emph{ and
}$H\left(  \theta_{1},\cdot\right)  =H_{1}$\emph{, where }$\theta_{0}$\emph{
and }$\theta_{1}$\emph{ are null vectors tangent to }$\mathbb{R}$\emph{ at
}$0$\emph{ and }$1$,$\ $\emph{respectively. We say that }$H_{0}$\emph{ and
}$H_{1}$\emph{ are }homotopic\emph{ and write }$H_{0}\sim H_{1}$\emph{. We say
that} $F:L^{\prime}\rightarrow L$ \emph{is a} \textit{homotopy equivalence
}\emph{if there is a homomorphism }$G:L\rightarrow L^{\prime}$ \emph{such
that} $G\circ F\sim\operatorname{id}_{L^{\prime}}$ \emph{and }$F\circ
G\sim\operatorname{id}_{L}\emph{.}$
\end{definition}

The homotopy $H:T\mathbb{R}\times L^{\prime}\rightarrow L$ determines a chain
homotopy operator (\cite{K4}, \cite{Balcerzak-Stokes})$\ $which implies that
$H_{0}^{\#}=H_{1}^{\#}:\mathsf{H}^{\bullet}\left(  L\right)  \rightarrow
\mathsf{H}^{\bullet}\left(  L^{\prime}\right)  $.

\begin{definition}
\emph{\cite{K5}} \label{defhomotop}\emph{Two Lie subalgebroids }$B_{0}%
,\,B_{1}\subset A$\emph{ (both over the same foliated manifold }$\left(
M,F\right)  $\emph{) are said to be }homotopic\emph{, if there exists a Lie
subalgebroid }$B\subset T\mathbb{R}\times A$\emph{ over }$\left(
\mathbb{R}\times M,T\mathbb{R}\times F\right)  $\emph{, such that for }%
$t\in\left\{  0,1\right\}  $%
\begin{equation}
\upsilon_{x}\in B_{t|x}\text{\ \ \emph{if\ and only if}\ \ }\left(  \theta
_{t},\upsilon_{x}\right)  \in B_{|\left(  t,x\right)  }. \label{homotopB}%
\end{equation}
$B$\emph{ is called a }subalgebroid joining $B_{0}$ with $B_{1}$.

\emph{See \cite[Proposition 5.2]{K5} to compare the relation of homotopic
subbundles of a principal bundle with the relation of homotopic
subalgebroids.}
\end{definition}

Let $B_{0},\,B_{1}$ be two homotopic Lie subalgebroids over $\left(
M,F\right)  $ and let $B\subset T\mathbb{R}\times A$ be a subalgebroid of
$T\mathbb{R}\times A$ joining $B_{0}$ with $B_{1}$, $t\in\left\{  0,1\right\}
$. For the homomorphism of Lie algebroids $F_{t}^{A}:A\rightarrow
T\mathbb{R}\times A$,\ $\upsilon_{x}\mapsto\left(  \theta_{t},\upsilon
_{x}\right)  $ over $f_{t}:M\rightarrow\mathbb{R}\times M$,\ $f_{t}\left(
x\right)  =\left(  t,x\right)  $,\ (\ref{homotopB}) yields $F_{t}^{A}\left[
B_{t}\right]  \subset B$. Applying the functoriality of $\Delta_{t\,\#}%
:=\Delta_{\left(  A,B_{t}\right)  \,\#}$\ and $\Delta_{\left(  A,B\right)
\,\#}$ (see Theorem \ref{functorof0}), we obtain the commutativity of the
diagram
\begin{center}
\xext=1600 \yext=1400
\begin{picture}(\xext, \yext)(\xoff,\yoff)
\setsqparms[1`-1`-1`1;990`550]
\putsquare(420,650)[\mathsf{H}^{\bullet}({\pmb{g}},B_{0})`\mathsf{H}^{\bullet
}(A)`\mathsf{H}^{\bullet}(0{\times}{\pmb{g}},B)`
\mathsf{H}^{\bullet}(T\mathbb{R}{\times}A);{\Delta}_{0 \#}`{F_{0}}%
^{+ \#}`{F_{0}}^{A \#}`{\Delta}_{(T\mathbb{R}{\times}A,B) \#}]
\setsqparms[0`1`1`1;1000`550]
\putsquare(420,100)[``\mathsf{H}^{\bullet}({\pmb{g}},B_{1})`\mathsf
{H}^{\bullet}(A);`{F_{1}}^{+ \#}`{F_{1}}^{A \#}`{\Delta}_{1 \#}]
\put(500,380){\makebox(0,0){$\simeq$}}
\put(500,900){\makebox(0,0){$\simeq$}}
\put(20,1210){\line(0,-1){1100}}
\put(19,1210){\line(1,0){195}}
\put(19,110){\vector(1,0){195}}
\put(-70,650){\makebox(0,0){$\alpha$}}
\end{picture}
\end{center}
where $F_{t}^{+\,\#}\equiv\left(  F_{t}^{A}\right)  ^{+\,\#}$. In the paper
\cite{K5} it is shown that $F_{t}^{+\,\#}$ are isomorphisms of algebras. We
add that in the proof of this fact one makes use of some theorem concerning
invariant cross-sections over $\mathbb{R}\times M$ (with respect to a suitable
representation) and one uses global solutions of some system of first-order
partial differential equations with parameters, see \cite{K6}.

For any flat $L$-connection $\nabla:L\rightarrow A$, the induced
$T\mathbb{R}\times L$-connection $\operatorname{id}_{T\mathbb{R}}\times\nabla$
is also flat.\ $F_{t}^{A}$ determines a homomorphism
\[
\left(  A,B_{t},\nabla\right)  \longrightarrow\left(  T\mathbb{R\times
}A,B,\operatorname{id}_{T\mathbb{R}}\times\nabla\right)
\]
of \textbf{ }FS-Lie algebroids over $f_{t}:M\rightarrow\mathbb{R}\times
M$;\ and so we can complete the previous diagram to the following one.
\begin{center}
\xext=3000 \yext=1370
\begin{picture}(\xext, \yext)(\xoff,\yoff)
\setsqparms[1`-1`-1`1;990`550]
\putsquare(560,650)[\mathsf{H}^{\bullet}({\pmb{g}},B_{0})`\mathsf{H}^{\bullet
}(A)`\mathsf{H}^{\bullet}(0{\times}{\pmb{g}},B)`
\mathsf{H}^{\bullet}(T\mathbb{R}{\times}A);{\Delta}_{0 \#}`{F_{0}}%
^{+ \#}`{F_{0}}^{A \#}`\ \ \ {\Delta}_{(T\mathbb{R}{\times}A,B) \#}]
\setsqparms[0`0`-1`1;800`550]
\putsquare(1840,650)[`\mathsf{H}^{\bullet}(L)``
\mathsf{H}^{\bullet}(T\mathbb{R}{\times}L);``{F_{0}}^{L \#}`({\operatorname
{id}{\times}\nabla})^{\#}]
\put(1740,1198){\vector(1,0){720}}
\put(2090,1280){\makebox(0,0){${\nabla}^{\#}$}}
\setsqparms[0`1`1`1;1000`550]
\putsquare(560,100)[``\mathsf{H}^{\bullet}({\pmb{g}},B_{1})`\mathsf
{H}^{\bullet}(A);`{F_{1}}^{+ \#}`{F_{1}}^{A \#}`{\Delta}_{1 \#}]
\setsqparms[0`0`1`0;800`550]
\putsquare(1840,100)[```\mathsf{H}^{\bullet}(L);``{F_{1}}^{L\#}`]
\put(1740,110){\vector(1,0){720}}
\put(2090,60){\makebox(0,0){${\nabla}^{\#}$}}
\put(640,380){\makebox(0,0){$\simeq$}}
\put(640,900){\makebox(0,0){$\simeq$}}
\put(160,1210){\line(0,-1){1100}}
\put(159,1210){\line(1,0){195}}
\put(159,110){\vector(1,0){195}}
\put(60,650){\makebox(0,0){$\alpha$}}
\end{picture}
\end{center}%

Observe that the rows of the above diagram are characteristic homomorphisms of
\textbf{ }FS-Lie algebroids. Since $F_{0}^{L},\,F_{1}^{L}:L\rightarrow
T\mathbb{R}\times L$ are homotopic homomorphisms, then $F_{0}^{L\,\#}%
=F_{1}^{L\,\#}$. To prove the homotopy independence of the exotic
characteristic homomorphism (in the sense of Definition \ref{defhomotop},
i.e., the independence of a class of homotopic subalgebroids), it is
sufficient to show that $F_{0}^{L\,\#},\,F_{1}^{L\,\#}$ are isomorphisms.

We shall see below that$\;F_{t}^{L}$\ are homotopically equivalent. Take the
projection $\pi:T\mathbb{R}\times L\rightarrow L$ (over $\operatorname{pr}%
_{2}:\mathbb{R}\times M\rightarrow M$). Of course, $\pi$ is a homomorphism of
Lie algebroids. Note that $F_{t_{o}}^{L}\circ\pi=\left(  \hat{t}_{o}\right)
_{\ast}\times\operatorname{id}_{L}$, where $\hat{t}_{o}:\mathbb{R\rightarrow
R}$ is defined by $t\mapsto t_{o}$. We take $\tau:\mathbb{R\times R\rightarrow
R}$,\ $\left(  s,t\right)  \mapsto t_{o}+s\,\left(  t-t_{o}\right)  $. Since
the differential $f_{\ast}:TM\rightarrow TN$ of any smooth mapping
$f:M\rightarrow N$ is a homomorphism of Lie algebroids \cite{K4}, we obtain
that the map $\tau_{\ast}:T\left(  \mathbb{R\times R}\right)
=T\mathbb{R\times}T\mathbb{R\rightarrow}T\mathbb{R}$ is a homomorphism of Lie
algebroids. We put
\begin{align*}
H  &  :T\mathbb{R\times}\left(  T\mathbb{R}\times L\right)  \longrightarrow
T\mathbb{R}\times L,\\
H  &  =\tau_{\ast}\times\operatorname{id}_{L}.
\end{align*}
Since
\begin{align*}
H\left(  \theta_{0},\cdot,\cdot\right)   &  =\tau\left(  0,\cdot\right)
_{\ast}\times\operatorname{id}_{L}=\left(  \hat{t}_{o}\right)  _{\ast}%
\times\operatorname{id}_{L}=F_{t_{o}}^{L}\circ\pi,\\
H\left(  \theta_{1},\cdot,\cdot\right)   &  =\tau\left(  \cdot,1\right)
_{\ast}\times\operatorname{id}_{L}=\operatorname{id}_{T\mathbb{R\times}L},
\end{align*}
$H$ is a homotopy joining $F_{t_{o}}^{L}\circ\pi$ with $\operatorname{id}%
_{T\mathbb{R\times}L}$,\ i.e. $F_{t_{o}}^{L}\circ\pi\sim\operatorname{id}%
_{T\mathbb{R\times}L}$. Evidently, $\pi\circ F_{t_{o}}^{L}=\operatorname{id}%
_{L}$. Therefore, $F_{t}^{L}$ are isomorphisms.

These facts lead us to the following result:

\begin{theorem}
[The Rigidity Theorem]\label{TheRigidityTheorem}If $B_{0},\,B_{1}\subset A$
are homotopic subalgebroids of $A$ and $\nabla:L\rightarrow A$ is a flat
$L$-connection in $A$, characteristic homomorphisms $\Delta_{\left(
A,B_{0},\nabla\right)  \#}:\mathsf{H}^{\bullet}\left(  {{\pmb{g},B}}%
_{0}\right)  \rightarrow\mathsf{H}^{\bullet}\left(  L\right)  $ and
$\Delta_{\left(  A,B_{1},\nabla\right)  \#}:\mathsf{H}^{\bullet}\left(
{{\pmb{g},B}}_{1}\right)  \rightarrow H_{L}\left(  M\right)  $ are equivalent
in the sense that there exists an isomorphism of algebras
\[
\alpha:\mathsf{H}^{\bullet}\left(  {{\pmb{g},B}}_{0}\right)  \overset{\simeq
}{\longrightarrow}\mathsf{H}^{\bullet}\left(  {{\pmb{g},B}}_{1}\right)
\]
such that
\[
\Delta_{\left(  A,B_{1},\nabla\right)  \#}\circ\alpha=\Delta_{\left(
A,B_{0},\nabla\right)  \#}.
\]
In particular, $\Delta_{\left(  A,B_{1}\right)  \#}\circ\alpha=\Delta_{\left(
A,B_{0}\right)  \#}$.
\end{theorem}

\begin{corollary}
Let $\mathfrak{f}$ be a vector bundle and $\mathcal{A}\left(  \mathfrak{f}%
\right)  $ its Lie algebroid. Two Lie subalgebroids $B_{0}=\mathcal{A}\left(
\mathfrak{f},\left\{  h_{0}\right\}  \right)  $, $B_{1}=\mathcal{A}\left(
\mathfrak{f},\left\{  h_{1}\right\}  \right)  $ of the Lie algebroid
$\mathcal{A}\left(  \mathfrak{f}\right)  $, corresponding to two Riemannian
metrics $h_{0}$, $h_{1}$, are homotopic Lie subalgebroids \emph{\cite{K5}}.
Therefore, according to the Rigidity Theorem \emph{\ref{TheRigidityTheorem}}
we conclude that $\Delta_{\left(  \mathcal{A}\left(  \mathfrak{f}\right)
,\mathcal{A}\left(  \mathfrak{f},\left\{  h_{0}\right\}  \right)  \right)
\#}=\Delta_{\left(  \mathcal{A}\left(  \mathfrak{f}\right)  ,\mathcal{A}%
\left(  \mathfrak{f},\left\{  h_{1}\right\}  \right)  \right)  \#}$, i.e. the
characteristic homomorphism for the pair $\left(  \mathcal{A}\left(
\mathfrak{f}\right)  ,\mathcal{A}\left(  \mathfrak{f},\left\{  h\right\}
\right)  \right)  $ is an intrinsic notion for $\mathcal{A}\left(
\mathfrak{f}\right)  $ not depending on the metric $h$.
\end{corollary}

\section{Particular Cases of the Universal Exotic Characteristic
Homomorphism\label{universal}}

\subsection{The Koszul Homomorphism\label{Koszul}}

In this section, we will consider the characteristic homomorphism
$\Delta_{\left(  \mathfrak{g},\mathfrak{h}\right)  \#}$ for a pair of Lie
algebras $\left(  \mathfrak{g},\mathfrak{h}\right)  $, $\mathfrak{h\subset g}%
$, and give a class of such pairs for which $\Delta_{\left(  \mathfrak{g}%
,\mathfrak{h}\right)  \#}$ is a monomorphism.

An arbitrary Lie algebra is a Lie algebroid over a point with the zero map as
an anchor. Considering the homomorphism of pairs of Lie algebras $\left(
\operatorname*{id}_{\mathfrak{g}},0\right)  :\left(  \mathfrak{g},0\right)
\rightarrow\left(  \mathfrak{g},\mathfrak{h}\right)  ,$ $\mathfrak{h}%
\subset\mathfrak{g}$, the functoriality property gives that
\[
\Delta_{\left(  \mathfrak{g},\mathfrak{h}\right)  \#}=\Delta_{\left(
\mathfrak{g},0\right)  \#}\circ\left(  \operatorname*{id}%
\nolimits_{\mathfrak{g}},0\right)  ^{+\#}=-\left(  \operatorname*{id}%
\nolimits_{\mathfrak{g}},0\right)  ^{+\#}:\mathsf{H}^{\bullet}\left(
\mathfrak{g,h}\right)  \rightarrow\mathsf{H}^{\bullet}\left(  \mathfrak{g}%
\right)
\]
is induced (in cohomology) by minus of the projection $s:\mathfrak{g}%
\rightarrow\mathfrak{g}/\emph{h}$, since $\Delta_{\left(  \mathfrak{g}%
,0\right)  \#}:\mathsf{H}^{\bullet}\left(  \mathfrak{g},0\right)
=\mathsf{H}^{\bullet}\left(  \mathfrak{g}\right)
\xrightarrow{{\left(-\operatorname*{id}\nolimits_{\mathfrak{g}}\right)_{\#}}}\mathsf{H}%
^{\bullet}\left(  \mathfrak{g}\right)  $. More precisely, if
$\operatorname*{k}:\left(  \bigwedge\mathfrak{g}^{\ast}\right)
_{i_{\mathfrak{h}}=0,\theta_{\mathfrak{h}}=0}\rightarrow\bigwedge
\mathfrak{g}^{\ast}$ denotes the inclusion from the basic subalgebra $\left(
\bigwedge\mathfrak{g}^{\ast}\right)  _{i_{\mathfrak{h}}=0,\theta
_{\mathfrak{h}}=0}$ (i.e. subalgebra of invariant and horizontal elements of
$\bigwedge\mathfrak{g}^{\ast}$ with respect to the Lie subalgebra
$\mathfrak{h}$) to $\bigwedge\mathfrak{g}^{\ast}$ (see \cite[p. 412]{GHV}),
then the secondary characteristic homomorphism $\Delta_{\left(  \mathfrak{g}%
,\mathfrak{h}\right)  \#}$ for the pair $\left(  \mathfrak{g},\mathfrak{h}%
\right)  $ can be written as a superposition%
\[
\begin{CD}
\Delta _{(\mathfrak{g},\mathfrak{h})\#}:\mathsf{H}^{\bullet}\left(  \mathfrak{g},\mathfrak{h}\right)
@> \ {\left(  -s \right) ^{\#}} >\cong>    \mathsf{H}^{\bullet}\left(  \mathfrak{g}/\mathfrak{h}\right)
@> \ {\operatorname{k}}^{\#} >>    \mathsf{H}^{\bullet}\left(  \mathfrak{g}\right),
\end{CD}
\]
where $\mathsf{H}^{\bullet}\left(  \mathfrak{g}/\mathfrak{h}\right)  $ denotes
the cohomology algebra $\mathsf{H}^{\bullet}\left(  \left(  \bigwedge
\mathfrak{g}^{\ast}\right)  _{i_{\mathfrak{h}}=0,\theta_{\mathfrak{h}}%
=0},d_{\mathfrak{g}}\right)  $.

\begin{example}
\label{examplereductivepair}\emph{Let }$\left(  \mathfrak{g},\mathfrak{h}%
\right)  $\emph{ be a reductive pair of Lie algebras (}$\mathfrak{h}%
\subset\mathfrak{g}$\emph{), }$s:\mathfrak{g}\rightarrow\mathfrak{g}%
/\mathfrak{h}$\emph{ } \emph{the canonical projection. Theorems IX and X from
\cite[sections 10.18, 10.19]{GHV} yield that }$\operatorname*{k}^{\#}$\emph{
is injective if and only if }$\mathsf{H}^{\bullet}\left(  \mathfrak{g}%
/\mathfrak{h}\right)  $\emph{ is generated by }$1$\emph{ and odd-degree
elements. Therefore, because of }$\left(  -s\right)  ^{\#}$ \emph{is an
isomorphism of algebras, it follows that }$\Delta_{\left(  \mathfrak{g}%
,\mathfrak{h}\right)  \#}$\emph{ is injective if and only if }$\mathsf{H}%
^{\bullet}\left(  \mathfrak{g},\mathfrak{h}\right)  $\emph{ is generated by
}$1$\emph{ and odd-degree elements. In a wide class of pairs of Lie algebras
}$\left(  \mathfrak{g},\mathfrak{h}\right)  $\emph{ such that }$\mathfrak{h}%
$\emph{ is reductive in }$\mathfrak{g}$\emph{, the homomorphism }%
$\operatorname*{k}^{\#}$\emph{ is injective if and only if }$\mathfrak{h}%
$\emph{ is noncohomologous to zero (briefly: n.c.z.) in }$\mathfrak{g}$\emph{
(i.e. if the homomorphism of algebras }$\mathsf{H}^{\bullet}\left(
\mathfrak{g}\right)  \rightarrow\mathsf{H}^{\bullet}\left(  \mathfrak{h}%
\right)  $\emph{ induced by the inclusion }$\mathfrak{h}\hookrightarrow
\mathfrak{g}$\emph{ is surjective). Tables I, II and III at the end of Section
XI \cite{GHV} contain many n.c.z. pairs, eg: }$\left(  \mathfrak{gl}\left(
n\right)  ,\mathfrak{so}\left(  n\right)  \right)  $\emph{ for odd }$n,$\emph{
}$\left(  \mathfrak{so}\left(  n,\mathbb{C}\right)  ,\mathfrak{so}\left(
k,\mathbb{C}\right)  \right)  $\emph{ for }$k<n,$\emph{ }$\left(
\mathfrak{so}\left(  2m+1\right)  ,\mathfrak{so}\left(  2k+1\right)  \right)
$\emph{ and }$\left(  \mathfrak{so}\left(  2m\right)  ,\mathfrak{so}\left(
2k+1\right)  \right)  $\emph{ for }$k<m$\emph{ and others.}
\end{example}

In view of the above, the examples below yield that the secondary
characteristic homomorphism for the reductive pair $\left(  \operatorname{End}%
\left(  V\right)  ,\operatorname{Sk}\left(  V\right)  \right)  $ of Lie
algebras is a monomorphism for any odd dimensional vector space $V$ and not a
monomorphism for even dimensional.

\begin{example}
[The pair of Lie algebras $\left(  \operatorname{End}\left(  V\right)
,\operatorname{Sk}\left(  V\right)  \right)  $]\label{ex}\emph{(a) Let }%
$V$\emph{ be a vector space of odd dimension, }$\dim V=2m-1$\emph{. Then }%
\[
\mathsf{H}^{\bullet}\hspace{-0.1cm}\left(  \operatorname{End}\hspace
{-0.1cm}\left(  V\right)  ,\operatorname{Sk}\hspace{-0.1cm}\left(  V\right)
\right)  \cong\mathsf{H}^{\bullet}\hspace{-0.1cm}\left(  \mathfrak{gl}\left(
2m-1,\mathbb{R}\right)  ,O\hspace{-0.1cm}\left(  2m-1\right)  \right)
\cong\bigwedge\hspace{-0.1cm}\left(  y_{1},y_{3},\ldots,y_{2m-1}\right)  ,
\]
\emph{where }$y_{2k-1}\in\mathsf{H}^{4k-3}\left(  \operatorname{End}\left(
V\right)  ,\operatorname{Sk}\left(  V\right)  \right)  $\emph{ are represented
by the multilinear trace forms (\cite{Godbillon}, \cite{K-T3}). We conclude
from the previous example that }$\Delta_{\left(  \operatorname{End}\left(
V\right)  ,\operatorname{Sk}\left(  V\right)  \right)  \#}$\emph{ is
injective}$.$\newline\emph{(b) In the case where }$V$ \emph{is an even
dimensional vector space (}$\dim V=2m$\emph{), we have \cite{Godbillon} }%
\[
\mathsf{H}^{\bullet}\hspace{-0.1cm}\left(  \operatorname{End}\hspace
{-0.1cm}\left(  V\right)  ,\operatorname{Sk}\hspace{-0.1cm}\left(  V\right)
\right)  \cong\mathsf{H}^{\bullet}\hspace{-0.1cm}\left(  \mathfrak{gl}%
\hspace{-0.1cm}\left(  2m,\mathbb{R}\right)  ,SO\hspace{-0.1cm}\left(
2m\right)  \right)  \cong\bigwedge\hspace{-0.1cm}\left(  y_{1},y_{3}%
,\ldots,y_{2m-1},y_{2m}\right)  ,
\]
\emph{where }$y_{2k-1}$\emph{ are the same as above and }$y_{2m}\in
\mathsf{H}^{2m}\left(  \operatorname{End}\left(  V\right)  ,\operatorname{Sk}%
\left(  V\right)  \right)  $\emph{ is a nonzero class determined by the
Pfaffian. For details concerning elements }$y_{2m}$ \emph{see \cite{B-K}.
Example \ref{examplereductivepair} shows that if }$\dim V$\emph{ is even}%
$,$\emph{ the homomorphism }$\Delta_{\left(  \operatorname{End}\left(
V\right)  ,\operatorname{Sk}\left(  V\right)  \right)  \#}$\emph{ is not a
monomorphism.}
\end{example}

\begin{example}
\label{exdirectproduct}\emph{Let} $\mathfrak{g}$, $\mathfrak{h}$ \emph{be Lie
algebras and} $\mathfrak{g}\oplus\mathfrak{h}$ \emph{their direct product.}
\emph{The characteristic homomorphism of the pair }$\left(  \mathfrak{g}%
\oplus\mathfrak{h},\mathfrak{h}\right)  $\emph{ is a monomorphism. It is equal
to }%
\begin{equation}
\Delta_{\left(  \mathfrak{g}\oplus\mathfrak{h},\mathfrak{h}\right)
\#}:\mathsf{H}^{\bullet}\left(  \mathfrak{g}\right)  \rightarrow
\mathsf{H}^{\bullet}\left(  \mathfrak{g}\right)  \otimes\mathsf{H}^{\bullet
}\left(  \mathfrak{h}\right)  ,\ \ \Delta_{\#\mathfrak{g}\oplus\mathfrak{h}%
,\mathfrak{h}}\left(  \left[  \Phi\right]  \right)  =[\left(  -1\right)
^{\left\vert \Phi\right\vert }\cdot\Phi]\otimes1.\nonumber
\end{equation}

\end{example}

\subsection{The Exotic Universal Characteristic Homomorphism of Principal
Fibre Subbundles\label{Exotic_for_principal_bundle}}

We recall briefly secondary (exotic) flat characteristic classes for flat
principal bundles \cite{K-T3} and its connection to the exotic characteristic
classes for a pair of a Lie algebroids of a suitable vector bundle and its
reduction \cite{B-K}.

Let $P$ be a $G$-principal fibre bundle on a smooth manifold $M$,
$\omega\subset TP$ a flat connection in $P$ and $P^{\prime}\subset P$ a
connected\ $H$-reduction, where $H\subset G$ is a closed Lie subgroup of $G$.
Let us consider Lie algebroids $\operatorname*{A}\left(  P\right)  ,$
$\operatorname*{A}\left(  P^{\prime}\right)  $ of the principal bundles
$P,P^{\prime}$, respectively, the induced flat connection $\omega
^{A}:TM\rightarrow\operatorname*{A}\left(  P\right)  $ in the Lie algebroid
$\operatorname*{A}\left(  P\right)  $, and the secondary characteristic
homomorphism%
\[
\Delta_{\left(  \operatorname*{A}\left(  P\right)  ,\operatorname*{A}\left(
P^{\prime}\right)  ,\omega^{A}\right)  \#}:\mathsf{H}^{\bullet}\left(
{{\pmb{g},}}\operatorname*{A}\left(  P^{\prime}\right)  \right)
\longrightarrow\mathsf{H}_{dR}^{\bullet}\left(  M\right)
\]
for the FS-Lie algebroid $\left(  \operatorname*{A}\left(  P\right)
,\operatorname*{A}\left(  P^{\prime}\right)  ,\omega^{A}\right)  $. Moreover,
let
\[
\Delta_{\left(  P,P^{\prime},\omega\right)  \#}:\mathsf{H}^{\bullet}\left(
\mathfrak{g},H\right)  \longrightarrow\mathsf{H}_{dR}^{\bullet}\left(
M\right)
\]
be the classical homomorphism on principal fibre bundles (F.~Kamber,
Ph.~Tondeur \cite{K-T3}), where $\mathsf{H}^{\bullet}(\mathfrak{g},H)$ is the
relative Lie algebra cohomology of $\left(  \mathfrak{g},H\right)  $ (see
\cite{K-T3}, \cite{Chev-Eil}). There exists an isomorphism of algebras
$\kappa:\mathsf{H}^{\bullet}\left(  \mathfrak{g}{,H}\right)  \overset{\simeq
}{\longrightarrow}\mathsf{H}^{\bullet}\left(  {{\pmb{g},}}\operatorname*{A}%
\left(  P^{\prime}\right)  \right)  $ such that%
\begin{equation}
\Delta_{\left(  \operatorname*{A}\left(  P\right)  ,\operatorname*{A}\left(
P^{\prime}\right)  ,\omega^{A}\right)  \#}\circ\kappa=\Delta_{\left(
P,P^{\prime},\omega\right)  \#} \label{equival}%
\end{equation}
(see \cite[Theorem 6.1]{K5}). Hence, their characteristic classes are
identical. In \cite{B-K} we showed that the homomorphism
\[
\Delta_{\left(  P,P^{\prime}\right)  \#}:=\Delta_{\left(  \operatorname*{A}%
\left(  P\right)  ,\operatorname*{A}\left(  P^{\prime}\right)  \right)
\#}\circ\kappa:\mathsf{H}^{\bullet}\left(  \mathfrak{g},H\right)
\longrightarrow\mathsf{H}^{\bullet}\left(  \operatorname*{A}\left(  P\right)
\right)  \longrightarrow\mathsf{H}_{dR}^{r\bullet}\left(  P\right)
\]
factorizes $\Delta_{\left(  P,P^{\prime},\omega\right)  \#}$ for any flat
connection $\omega$ in $P$, i.e. the following diagram commutes
\begin{center}
\settriparms[-1`1`1;]
\Atriangle[\mathsf{H}^{\bullet}_{dR}(P)`\mathsf{H}^{\bullet}({\mathfrak{g}%
},H)`\mathsf{H}^{\bullet}_{dR}(M),;{\Delta_{(P,P^{\prime})\#}}`
{\omega}^{\#}`{\Delta_{(P,P^{\prime},\omega)\#}}]
\end{center}
where $\omega^{\#}$ on the level of right-invariant differential forms
${\Omega}^{r}\left(  P\right)  $ is given as the pullback of differential
forms. In particular, if $G$ is a compact, connected Lie group and $P^{\prime
}$ is a connected $H$-reduction in a $G$-principal bundle $P$, $H\subset G$,
then there exists a homomorphism of algebras%
\[
\Delta_{\left(  P,P^{\prime}\right)  \#}:\mathsf{H}^{\bullet}\left(
\mathfrak{g},H\right)  \longrightarrow\mathsf{H}_{dR}^{\bullet}\left(
P\right)
\]
(called a \textit{universal exotic characteristic homomorphism} for the pair
$P^{\prime}\subset P$) such that for arbitrary flat connection $\omega$ in
$P$, the characteristic homomorphism $\Delta_{\left(  P,P^{\prime}%
,\omega\right)  \#}:\mathsf{H}^{\bullet}\left(  \mathfrak{g},H\right)
\rightarrow\mathsf{H}_{dR}^{\bullet}\left(  M\right)  $ is factorized by
$\Delta_{\left(  P,P^{\prime}\right)  \#}$, i.e. the following diagram is
commutative
\begin{center}
\settriparms[-1`1`1;]
\Atriangle[\mathsf{H}^{\bullet}_{dR}(P)`\mathsf{H}^{\bullet}({\mathfrak{g}%
},H)`\mathsf{H}^{\bullet}_{dR}(M).;{\Delta_{(P,P^{\prime})\#}}`{\omega}%
^{\#}`{\Delta_{(P,P^{\prime},\omega)\#}}]
\end{center}%

\section{About a Monomorphicity of the Universal Exotic Characteristic
Homomorphism for a Pair of Transitive Lie Algebroids\label{ostatni}}

Consider a pair $\left(  A,B\right)  $ of transitive Lie algebroids on a
manifold $M$, $B\subset A$, $x\in M$, and a pair of adjoint Lie algebras
$\left(  {{\pmb{g}}}_{x},{{\pmb{h}}}_{x}\right)  $. Clearly, the inclusion
$\iota_{x}:\left(  {{\pmb{g}}}_{x},{{\pmb{h}}}_{x}\right)  \rightarrow\left(
A,B\right)  $ is a homomorphism of pairs of Lie algebroids over $\left\{
\ast\right\}  \hookrightarrow M$. Theorem \ref{functorof0} gives rise to the
commutative diagram
\begin{equation}%
\begin{split}%
\setsqparms[1`1`1`1;900`500]
\square[\mathsf{H}^{\bullet}({\pmb{g}},B)`\mathsf{H}^{\bullet}(A)`\mathsf
{H}^{\bullet}({\pmb{g}_{x}},\pmb{h}_{x})`\mathsf{H}^{\bullet}(\pmb{g}_{x}).;
{\Delta}_{(A,B)\#} `\iota_{x}^{+ \#}`\iota_{x}^{\#}`{\Delta}_{(\pmb{g}%
_{x},\pmb{h}_{x})\#}]%
\end{split}
\label{fu}%
\end{equation}
Obviously, if the left and lower homomorphisms are monomorphisms, then
$\Delta_{\left(  A,B\right)  \#}$ is a monomorphism as well. The homomorphism
$\iota_{x}^{+\#}$ is a monomorphism, if each invariant element $v\in\left(
\bigwedge\left(  {{\pmb{g}}}_{x}{{/\pmb{h}}}_{x}\right)  ^{\ast}\right)
^{\mathfrak{h}_{x}}$ can be extended to a global invariant cross-section of
the vector bundle $\bigwedge\left(  {{\pmb{g}/\pmb{h}}}\right)  ^{\ast}$. In
consequence, we obtain the following theorem linking the Koszul homomorphism
with exotic characteristic classes:

\begin{theorem}
\label{threductive_general}Let $\left(  A,B\right)  $ be a pair of transitive
Lie algebroids over a manifold $M$, $B\subset A$, and let $\left(  {{\pmb{g}}%
}_{x},{{\pmb{h}}}_{x}\right)  $ be a pair of adjoint Lie algebras at $x\in M$.
If each element of $\left(  \bigwedge\left(  {{\pmb{g}}}_{x}{{/\pmb{h}}}%
_{x}\right)  ^{\ast}\right)  ^{\mathfrak{h}_{x}}$ can be extended to an
invariant cross-section of $\bigwedge\left(  {{\pmb{g}/\pmb{h}}}\right)
^{\ast}$ and the Koszul homomorphism $\Delta_{\left(  {{\pmb{g}}}%
_{x},{{\pmb{h}}}_{x}\right)  \#}$ \ for the pair $\left(  {{\pmb{g}}}%
_{x},{{\pmb{h}}}_{x}\right)  $ is a monomorphism, then $\Delta_{\left(
A,B\right)  \#}$ is a monomorphism.
\end{theorem}

The assumptions of the above theorem hold for integrable Lie algebroids $A$
and $B$ ($B\subset A$), i.e. if $A=\operatorname*{A}\left(  P\right)  $ for
some principal $G$-bundle $P$ and $B=\operatorname*{A}\left(  P^{\prime
}\right)  $ for some reduction $P^{\prime}$ of $P$ with connected structural
Lie group $H\subset G$ (remark: on account of Theorem 1.1 in \cite{exp} for
any transitive Lie subalgebroid $B\subset\operatorname*{A}\left(  P\right)  $
there exists a connected reduction $P^{\prime}$ of $P$ having $B$ as its Lie
algebroid, i.e. $B=\operatorname*{A}\left(  P^{\prime}\right)  $, but, in
general, the structural Lie group of $P^{\prime}$ may be not connected). Let
$\mathfrak{g}$ and $\mathfrak{h}$ denote the Lie algebras of $G$ and $H$,
respectively. The representation $\operatorname*{ad}_{B,{{\pmb{h}}}}$ is
integrable: it is a differential of the representation $\operatorname*{Ad}%
_{P^{\prime},{{\pmb{h}}}}:P^{\prime}\rightarrow L\left(  {{\pmb{g}/\pmb{h}}%
}\right)  $ of the principal fibre bundle $P^{\prime}$ defined by
$z\mapsto\left[  \hat{z}\right]  $, see \cite[p. 218]{K5}. We recall that for
each $z\in P^{\prime}$, the isomorphism $\hat{z}:\mathfrak{g}\rightarrow
\pmb{g}_{x},$ $v\mapsto\left[  \left(  A_{z}\right)  _{\ast e}v\right]  $
($A_{z}:G\rightarrow P,$ $a\mapsto za$) maps $\mathfrak{h}$ onto $\pmb{h}_{x}$
(see \cite[Sec. 5.1]{K1}) and determines an isomorphism $\left[  \hat
{z}\right]  :\mathfrak{g}/\mathfrak{h}\rightarrow\pmb{g}_{x}/\pmb{h}_{x}$.
Therefore (see also (\cite[Prop. 5.5.2-3]{K1})), we have a natural isomorphism%
\[
\kappa:\left(  \bigwedge\left(  {{\pmb{g}}}_{x}{{/\pmb{h}}}_{x}\right)
^{\ast}\right)  ^{H}\overset{\cong}{\longrightarrow}\left(  \Gamma\left(
\bigwedge\left(  {{\pmb{g}/\pmb{h}}}\right)  ^{\ast}\right)  \right)
^{\Gamma\left(  B\right)  }%
\]
and because of the connectedness of $H$, $\hspace{-0.05cm}\left(
\bigwedge\left(  {{\pmb{g}}}_{x}{{/\pmb{h}}}_{x}\right)  ^{\ast}\right)
^{H}\hspace{-0.06cm}=\hspace{-0.05cm}\left(  \bigwedge\left(  {{\pmb{g}}}%
_{x}{{/\pmb{h}}}_{x}\right)  ^{\ast}\right)  ^{\mathfrak{h}_{x}}$, which gives
that $\iota_{x}^{+\#}$ is an isomorphism. In this way we obtain the following corollary:

\begin{corollary}
Let $\left(  A,B\right)  $ be a pair of Lie algebroids, $B\subset A$, where
$A$ is an integrable Lie algebroid via a principal fibre bundle $P$ and let
$\left(  {{\pmb{g}}}_{x},{{\pmb{h}}}_{x}\right)  $ be a pair of adjoint Lie
algebras at $x\in M$. If the structure Lie group of the connected reduction
$P^{\prime}$ of $P$ such that $\operatorname*{A}\left(  P^{\prime}\right)  =B$
is a connected Lie group and the Koszul homomorphism $\Delta_{\left(
{{\pmb{g}}}_{x},{{\pmb{h}}}_{x}\right)  \#}$ \ for the pair $\left(
{{\pmb{g}}}_{x},{{\pmb{h}}}_{x}\right)  $ is a monomorphism \emph{(}for
examples see Example \emph{\ref{examplereductivepair})}, then $\Delta_{\left(
A,B\right)  \#}$ is a monomorphism as well.
\end{corollary}

Let $\left(  A,B\right)  $ be a pair of transitive Lie algebroids over a
manifold $M$, $B\subset A$, for which the kernels ${{\pmb{g}}},{{\pmb{h}}}$ of
anchors are abelian Lie algebra bundles, and let $x\in M$. An example of the
mentioned Lie algebroid is the Lie algebroid $A\left(  G,H\right)  $ of a
nonclosed and connected Lie subgroup $H$ of $G$ (see \cite{K1}). The existence
of such nonintegrable Lie algebroids is shown in \cite{A-M}. Since the
homomorphism $\mathsf{H}^{\bullet}\left(  {{\pmb{g}}}_{x}\right)
=\bigwedge\left(  {{\pmb{g}}}_{x}\right)  ^{\ast}\rightarrow\bigwedge\left(
{{\pmb{h}}}_{x}\right)  ^{\ast}=\mathsf{H}^{\bullet}\left(  {{\pmb{h}}}%
_{x}\right)  $ induced by the inclusion ${{\pmb{g}}}_{x}\hookrightarrow
{{\pmb{h}}}_{x}$ is surjective and $\left(  {{\pmb{g}}}_{x},{{\pmb{h}}}%
_{x}\right)  $ is a reductive pair of Lie algebras, the Koszul homomorphism
$\Delta_{\left(  {{\pmb{g}}}_{x},{{\pmb{h}}}_{x}\right)  \#}$ is injective
(see Example \ref{examplereductivepair}). Hence, in view of Theorem
\ref{threductive_general} we obtain that if $\left(  A,B\right)  $ is a pair
of transitive Lie algebroids on a manifold $M$ such that kernels of their
anchors ${{\pmb{g}}},{{\pmb{h}}}$ are abelian Lie algebra bundles and each
element of $\left(  \bigwedge\left(  {{\pmb{g}}}_{x}{{/\pmb{h}}}_{x}\right)
^{\ast}\right)  ^{\mathfrak{h}_{x}}$ can be extended to an invariant
cross-section of $\bigwedge\left(  {{\pmb{g}/\pmb{h}}}\right)  ^{\ast}$ for
all $x\in M$, then $\Delta_{\left(  A,B\right)  \#}$ is a monomorphism.

\begin{remark}
\emph{An example of a nontrivial universal characteristic class determined by
the Pfaffian (see \cite{B-K}) shows that there exists a pair of transitive and
integrable Lie algebroids }$\left(  A,B\right)  $\emph{ for which the left
arrow in the diagram (\ref{fu}) describes an isomorphism, the bottom one is
not a monomorphism, however, the top one is a monomorphism.}
\end{remark}

\ \bigskip

Bogdan Balcerzak

e-mail: bogdan.balcerzak@p.lodz.pl

\bigskip

Jan Kubarski

e-mail: jan.kubarski@p.lodz.pl

\bigskip

Institute of Mathematics, Technical University of \L \'{o}d\'{z}

ul. W\'{o}lcza\'{n}ska 215, 90-924 \L \'{o}d\'{z}, Poland

\end{document}

%% file: diagram.tex
\makeatletter
 
\def\diagram{\m@th\leftwidth=\z@ \rightwidth=\z@ \topheight=\z@
\botheight=\z@ \setbox\@picbox\hbox\bgroup}
 
\def\enddiagram{\egroup\wd\@picbox\rightwidth\unitlength
\ht\@picbox\topheight\unitlength \dp\@picbox\botheight\unitlength
\hskip\leftwidth\unitlength\box\@picbox}
 
\def\bfig{\begin{diagram}}
\def\efig{\end{diagram}}
\newcount\wideness \newcount\leftwidth \newcount\rightwidth
\newcount\highness \newcount\topheight \newcount\botheight
 
\def\ratchet#1#2{\ifnum#1<#2 \global #1=#2 \fi}
 
\def\putbox(#1,#2)#3{%
\horsize{\wideness}{#3} \divide\wideness by 2
{\advance\wideness by #1 \ratchet{\rightwidth}{\wideness}}
{\advance\wideness by -#1 \ratchet{\leftwidth}{\wideness}}
\vertsize{\highness}{#3} \divide\highness by 2
{\advance\highness by #2 \ratchet{\topheight}{\highness}}
{\advance\highness by -#2 \ratchet{\botheight}{\highness}}
\put(#1,#2){\makebox(0,0){$#3$}}}
 
\def\putlbox(#1,#2)#3{%
\horsize{\wideness}{#3}
{\advance\wideness by #1 \ratchet{\rightwidth}{\wideness}}
{\ratchet{\leftwidth}{-#1}}
\vertsize{\highness}{#3} \divide\highness by 2
{\advance\highness by #2 \ratchet{\topheight}{\highness}}
{\advance\highness by -#2 \ratchet{\botheight}{\highness}}
\put(#1,#2){\makebox(0,0)[l]{$#3$}}}
 
\def\putrbox(#1,#2)#3{%
\horsize{\wideness}{#3}
{\ratchet{\rightwidth}{#1}}
{\advance\wideness by -#1 \ratchet{\leftwidth}{\wideness}}
\vertsize{\highness}{#3} \divide\highness by 2
{\advance\highness by #2 \ratchet{\topheight}{\highness}}
{\advance\highness by -#2 \ratchet{\botheight}{\highness}}
\put(#1,#2){\makebox(0,0)[r]{$#3$}}}

\def\adjust[#1]{} 
 
\newcount \coefa
\newcount \coefb
\newcount \coefc
\newcount\tempcounta
\newcount\tempcountb
\newcount\tempcountc
\newcount\tempcountd
\newcount\xext
\newcount\yext
\newcount\xoff
\newcount\yoff
\newcount\gap%
\newcount\arrowtypea
\newcount\arrowtypeb
\newcount\arrowtypec
\newcount\arrowtyped
\newcount\arrowtypee
\newcount\height
\newcount\width
\newcount\xpos
\newcount\ypos
\newcount\run
\newcount\rise
\newcount\arrowlength
\newcount\halflength
\newcount\arrowtype
\newdimen\tempdimen
\newdimen\xlen
\newdimen\ylen
\newsavebox{\tempboxa}%
\newsavebox{\tempboxb}%
\newsavebox{\tempboxc}%
 
\newdimen\w@dth
 
\def\setw@dth#1#2{\setbox\z@\hbox{\m@th$#1$}\w@dth=\wd\z@
\setbox\@ne\hbox{\m@th$#2$}\ifnum\w@dth<\wd\@ne \w@dth=\wd\@ne \fi
\advance\w@dth by 1.2em}
 
\def\t@^#1_#2{\allowbreak\def\n@one{#1}\def\n@two{#2}\mathrel
{\setw@dth{#1}{#2}
\mathop{\hbox to \w@dth{\rightarrowfill}}\limits
\ifx\n@one\empty\else ^{\box\z@}\fi
\ifx\n@two\empty\else _{\box\@ne}\fi}}
\def\t@@^#1{\@ifnextchar_{\t@^{#1}}{\t@^{#1}_{}}}
\def\to{\@ifnextchar^{\t@@}{\t@@^{}}}
 
\def\t@left^#1_#2{\def\n@one{#1}\def\n@two{#2}\mathrel{\setw@dth{#1}{#2}
\mathop{\hbox to \w@dth{\leftarrowfill}}\limits
\ifx\n@one\empty\else ^{\box\z@}\fi
\ifx\n@two\empty\else _{\box\@ne}\fi}}
\def\t@@left^#1{\@ifnextchar_{\t@left^{#1}}{\t@left^{#1}_{}}}
\def\toleft{\@ifnextchar^{\t@@left}{\t@@left^{}}}
 
\def\two@^#1_#2{\allowbreak
\def\n@one{#1}\def\n@two{#2}\mathrel{\setw@dth{#1}{#2}
\mathop{\vcenter{\lineskip\z@\baselineskip\z@
                 \hbox to \w@dth{\rightarrowfill}%
                 \hbox to \w@dth{\rightarrowfill}}%
       }\limits
\ifx\n@one\empty\else ^{\box\z@}\fi
\ifx\n@two\empty\else _{\box\@ne}\fi}}
\def\tw@@^#1{\@ifnextchar _{\two@^{#1}}{\two@^{#1}_{}}}
\def\two{\@ifnextchar ^{\tw@@}{\tw@@^{}}}
 
\def\tofr@^#1_#2{\def\n@one{#1}\def\n@two{#2}\mathrel{\setw@dth{#1}{#2}
\mathop{\vcenter{\hbox to \w@dth{\rightarrowfill}\kern-1.7ex
                 \hbox to \w@dth{\leftarrowfill}}%
       }\limits
\ifx\n@one\empty\else ^{\box\z@}\fi
\ifx\n@two\empty\else _{\box\@ne}\fi}}
\def\t@fr@^#1{\@ifnextchar_ {\tofr@^{#1}}{\tofr@^{#1}_{}}}
\def\tofro{\@ifnextchar^ {\t@fr@}{\t@fr@^{}}}

\def\mon{\mathop{\m@th\hbox to
      14.6\P@{\lasyb\char'51\hskip-2.1\P@$\arrext$\hss
$\mathord\rightarrow$}}\limits} 
\def\leftmono{\mathrel{\m@th\hbox to
14.6\P@{$\mathord\leftarrow$\hss$\arrext$\hskip-2.1\P@\lasyb\char'50%
}}\limits} 
\mathchardef\arrext="0200       

\def\settypes(#1,#2,#3){\arrowtypea#1 \arrowtypeb#2 \arrowtypec#3}
\def\settoheight#1#2{\setbox\@tempboxa\hbox{#2}#1\ht\@tempboxa\relax}%
\def\settodepth#1#2{\setbox\@tempboxa\hbox{#2}#1\dp\@tempboxa\relax}%
\def\settokens`#1`#2`#3`#4`{%
     \def\tokena{#1}\def\tokenb{#2}\def\tokenc{#3}\def\tokend{#4}}
\def\setsqparms[#1`#2`#3`#4;#5`#6]{%
\arrowtypea #1
\arrowtypeb #2
\arrowtypec #3
\arrowtyped #4
\width #5
\height #6
}
\def\setpos(#1,#2){\xpos=#1 \ypos#2}

\def\settriparms[#1`#2`#3;#4]{\settripairparms[#1`#2`#3`1`1;#4]}%
 
\def\settripairparms[#1`#2`#3`#4`#5;#6]{%
\arrowtypea #1
\arrowtypeb #2
\arrowtypec #3
\arrowtyped #4
\arrowtypee #5
\width #6
\height #6
}
 
\def\resetparms{\settripairparms[1`1`1`1`1;500]\width 500}
 
\resetparms
 
\def\mvector(#1,#2)#3{
\put(0,0){\vector(#1,#2){#3}}%
\put(0,0){\vector(#1,#2){26}}%
}
\def\evector(#1,#2)#3{{
\arrowlength #3
\put(0,0){\vector(#1,#2){\arrowlength}}%
\advance \arrowlength by-30
\put(0,0){\vector(#1,#2){\arrowlength}}%
}}
 
\def\horsize#1#2{%
\settowidth{\tempdimen}{$#2$}%
#1=\tempdimen
\divide #1 by\unitlength
}
 
\def\vertsize#1#2{%
\settoheight{\tempdimen}{$#2$}%
#1=\tempdimen
\settodepth{\tempdimen}{$#2$}%
\advance #1 by\tempdimen
\divide #1 by\unitlength
}
 
\def\putvector(#1,#2)(#3,#4)#5#6{{%
\ifnum3<\arrowtype
\putdashvector(#1,#2)(#3,#4)#5\arrowtype
\else
\ifnum\arrowtype<-3
\putdashvector(#1,#2)(#3,#4)#5\arrowtype
\else
\xpos=#1
\ypos=#2
\run=#3
\rise=#4
\arrowlength=#5
\ifnum \arrowtype<0
    \ifnum \run=0
        \advance \ypos by-\arrowlength
    \else
        \tempcounta \arrowlength
        \multiply \tempcounta by\rise
        \divide \tempcounta by\run
        \ifnum\run>0
            \advance \xpos by\arrowlength
            \advance \ypos by\tempcounta
        \else
            \advance \xpos by-\arrowlength
            \advance \ypos by-\tempcounta
        \fi
    \fi
    \multiply \arrowtype by-1
    \multiply \rise by-1
    \multiply \run by-1
\fi
\ifcase \arrowtype
\or \put(\xpos,\ypos){\vector(\run,\rise){\arrowlength}}%
\or \put(\xpos,\ypos){\mvector(\run,\rise)\arrowlength}%
\or \put(\xpos,\ypos){\evector(\run,\rise){\arrowlength}}%
\fi\fi\fi
}}
 
\def\putsplitvector(#1,#2)#3#4{
\xpos #1
\ypos #2
\arrowtype #4
\halflength #3
\arrowlength #3
\gap 140
\advance \halflength by-\gap
\divide \halflength by2
\ifnum\arrowtype>0
   \ifcase \arrowtype
   \or \put(\xpos,\ypos){\line(0,-1){\halflength}}%
       \advance\ypos by-\halflength
       \advance\ypos by-\gap
       \put(\xpos,\ypos){\vector(0,-1){\halflength}}%
   \or \put(\xpos,\ypos){\line(0,-1)\halflength}%
       \put(\xpos,\ypos){\vector(0,-1)3}%
       \advance\ypos by-\halflength
       \advance\ypos by-\gap
       \put(\xpos,\ypos){\vector(0,-1){\halflength}}%
   \or \put(\xpos,\ypos){\line(0,-1)\halflength}%
       \advance\ypos by-\halflength
       \advance\ypos by-\gap
       \put(\xpos,\ypos){\evector(0,-1){\halflength}}%
   \fi
\else \arrowtype=-\arrowtype
   \ifcase\arrowtype
   \or \advance \ypos by-\arrowlength
       \put(\xpos,\ypos){\line(0,1){\halflength}}%
       \advance\ypos by\halflength
       \advance\ypos by\gap
       \put(\xpos,\ypos){\vector(0,1){\halflength}}%
   \or \advance \ypos by-\arrowlength
       \put(\xpos,\ypos){\line(0,1)\halflength}%
       \put(\xpos,\ypos){\vector(0,1)3}%
       \advance\ypos by\halflength
       \advance\ypos by\gap
       \put(\xpos,\ypos){\vector(0,1){\halflength}}%
   \or \advance \ypos by-\arrowlength
       \put(\xpos,\ypos){\line(0,1)\halflength}%
       \advance\ypos by\halflength
       \advance\ypos by\gap
       \put(\xpos,\ypos){\evector(0,1){\halflength}}%
   \fi
\fi
}
 
\def\putmorphism(#1)(#2,#3)[#4`#5`#6]#7#8#9{{%
\run #2
\rise #3
\ifnum\rise=0
  \puthmorphism(#1)[#4`#5`#6]{#7}{#8}#9%
\else\ifnum\run=0
  \putvmorphism(#1)[#4`#5`#6]{#7}{#8}#9%
\else
\setpos(#1)%
\arrowlength #7
\arrowtype #8
\ifnum\run=0
\else\ifnum\rise=0
\else
\ifnum\run>0
    \coefa=1
\else
   \coefa=-1
\fi
\ifnum\arrowtype>0
   \coefb=0
   \coefc=-1
\else
   \coefb=\coefa
   \coefc=1
   \arrowtype=-\arrowtype
\fi
\width=2
\multiply \width by\run
\divide \width by\rise
\ifnum \width<0  \width=-\width\fi
\advance\width by60
\if l#9 \width=-\width\fi
\putbox(\xpos,\ypos){#4}
{\multiply \coefa by\arrowlength
\advance\xpos by\coefa
\multiply \coefa by\rise
\divide \coefa by\run
\advance \ypos by\coefa
\putbox(\xpos,\ypos){#5} }%
{\multiply \coefa by\arrowlength
\divide \coefa by2
\advance \xpos by\coefa
\advance \xpos by\width
\multiply \coefa by\rise
\divide \coefa by\run
\advance \ypos by\coefa
\if l#9%
   \putrbox(\xpos,\ypos){#6}%
\else\if r#9%
   \putlbox(\xpos,\ypos){#6}%
\fi\fi }%
{\multiply \rise by-\coefc
\multiply \run by-\coefc
\multiply \coefb by\arrowlength
\advance \xpos by\coefb
\multiply \coefb by\rise
\divide \coefb by\run
\advance \ypos by\coefb
\multiply \coefc by70
\advance \ypos by\coefc
\multiply \coefc by\run
\divide \coefc by\rise
\advance \xpos by\coefc
\multiply \coefa by140
\multiply \coefa by\run
\divide \coefa by\rise
\advance \arrowlength by\coefa
\ifcase\arrowtype
\or \put(\xpos,\ypos){\vector(\run,\rise){\arrowlength}}%
\or \put(\xpos,\ypos){\mvector(\run,\rise){\arrowlength}}%
\or \put(\xpos,\ypos){\evector(\run,\rise){\arrowlength}}%
\fi}\fi\fi\fi\fi}}

\newcount\numbdashes \newcount\lengthdash \newcount\increment
 
\def\howmanydashes{
\numbdashes=\arrowlength \lengthdash=40
\divide\numbdashes by \lengthdash
\lengthdash=\arrowlength
\divide\lengthdash by \numbdashes
\increment=\lengthdash
\multiply\lengthdash by 3
\divide\lengthdash by 5
}
 
\def\putdashvector(#1)(#2,#3)#4#5{%
\ifnum#3=0 \putdashhvector(#1){#4}#5
\else
\ifnum#2=0
\putdashvvector(#1){#4}#5\fi\fi}
 
\def\putdashhvector(#1,#2)#3#4{{%
\arrowlength=#3 \howmanydashes
\multiput(#1,#2)(\increment,0){\numbdashes}%
{\vrule height .4pt width \lengthdash\unitlength}
\arrowtype=#4 \xpos=#1
\ifnum\arrowtype<0 \advance\arrowtype by 7 \fi
\ifcase\arrowtype
\or \advance\xpos by 10
    \put(\xpos,#2){\vector(-1,0){\lengthdash}}
    \advance\xpos by 40
    \put(\xpos,#2){\vector(-1,0){\lengthdash}}
\or \advance \xpos by 10
    \put(\xpos,#2){\vector(-1,0){\lengthdash}}
    \advance\xpos by  \arrowlength
    \advance\xpos by  -50
    \put(\xpos,#2){\vector(-1,0){\lengthdash}}
\or \advance\xpos by 10
    \put(\xpos,#2){\vector(-1,0){\lengthdash}}
\or \advance\xpos by \arrowlength
    \advance\xpos by -\lengthdash
    \put(\xpos,#2){\vector(1,0){\lengthdash}}
\or {\advance\xpos by 10
    \put(\xpos,#2){\vector(1,0){\lengthdash}}}
    \advance\xpos by \arrowlength
    \advance\xpos by -\lengthdash
    \put(\xpos,#2){\vector(1,0){\lengthdash}}
\or \advance\xpos by \arrowlength
    \advance\xpos by -\lengthdash
    \put(\xpos,#2){\vector(1,0){\lengthdash}}
    \advance\xpos by -40
    \put(\xpos,#2){\vector(1,0){\lengthdash}}
   \fi
}}
 
\def\putdashvvector(#1,#2)#3#4{{%
\arrowlength=#3 \howmanydashes
\ypos=#2 \advance\ypos by -\arrowlength
\multiput(#1,#2)(0,\increment){\numbdashes}%
    {\vrule width .4pt height \lengthdash\unitlength}
\arrowtype=#4 \ypos=#2
\ifnum\arrowtype<0 \advance\arrowtype by 7 \fi
\ifcase\arrowtype
\or \advance\ypos by \arrowlength \advance\ypos by -40
    \put(#1,\ypos){\vector(0,1){\lengthdash}}
    \advance\ypos by -40
    \put(#1,\ypos){\vector(0,1){\lengthdash}}
\or \advance\ypos by 10
    \put(#1,\ypos){\vector(0,1){\lengthdash}}
    \advance\ypos by \arrowlength \advance\ypos by -40
    \put(#1,\ypos){\vector(0,1){\lengthdash}}
\or \advance\ypos by \arrowlength \advance\ypos by -40
    \put(#1,\ypos){\vector(0,1){\lengthdash}}
\or \advance\ypos by 10
    \put(#1,\ypos){\vector(0,-1){\lengthdash}}
\or \advance\ypos by 10
    \put(#1,\ypos){\vector(0,-1){\lengthdash}}
    \advance\ypos by \arrowlength \advance\ypos by -40
    \put(#1,\ypos){\vector(0,-1){\lengthdash}}
\or \advance\ypos by 10
    \put(#1,\ypos){\vector(0,-1){\lengthdash}}
    \advance\ypos by 40
    \put(#1,\ypos){\vector(0,-1){\lengthdash}}
\fi
}}
 
\def\puthmorphism(#1,#2)[#3`#4`#5]#6#7#8{{%
\xpos #1
\ypos #2
\width #6
\arrowlength #6
\arrowtype=#7
\putbox(\xpos,\ypos){#3\vphantom{#4}}%
{\advance \xpos by\arrowlength
\putbox(\xpos,\ypos){\vphantom{#3}#4}}%
\horsize{\tempcounta}{#3}%
\horsize{\tempcountb}{#4}%
\divide \tempcounta by2
\divide \tempcountb by2
\advance \tempcounta by30
\advance \tempcountb by30
\advance \xpos by\tempcounta
\advance \arrowlength by-\tempcounta
\advance \arrowlength by-\tempcountb
\putvector(\xpos,\ypos)(1,0)\arrowlength\arrowtype
\divide \arrowlength by2
\advance \xpos by\arrowlength
\vertsize{\tempcounta}{#5}%
\divide\tempcounta by2
\advance \tempcounta by20
\if a#8 %
   \advance \ypos by\tempcounta
   \putbox(\xpos,\ypos){#5}%
\else
   \advance \ypos by-\tempcounta
   \putbox(\xpos,\ypos){#5}%
\fi}}
 
\def\putvmorphism(#1,#2)[#3`#4`#5]#6#7#8{{%
\xpos #1
\ypos #2
\arrowlength #6
\arrowtype #7
\settowidth{\xlen}{$#5$}%
\putbox(\xpos,\ypos){#3}%
{\advance \ypos by-\arrowlength
\putbox(\xpos,\ypos){#4}}%
{\advance\arrowlength by-140
\advance \ypos by-70
\ifdim\xlen>0pt
   \if m#8%
      \putsplitvector(\xpos,\ypos)\arrowlength\arrowtype
   \else
   \putvector(\xpos,\ypos)(0,-1)\arrowlength\arrowtype
   \fi
\else
   \putvector(\xpos,\ypos)(0,-1)\arrowlength\arrowtype
\fi}%
\ifdim\xlen>0pt
   \divide \arrowlength by2
   \advance\ypos by-\arrowlength
   \if l#8%
      \advance \xpos by-40
      \putrbox(\xpos,\ypos){#5}%
   \else\if r#8%
      \advance \xpos by40
      \putlbox(\xpos,\ypos){#5}%
   \else
      \putbox(\xpos,\ypos){#5}%
   \fi\fi
\fi
}}
 
\def\putsquarep<#1>(#2)[#3;#4`#5`#6`#7]{{%
\setsqparms[#1]%
\setpos(#2)%
\settokens`#3`%
\puthmorphism(\xpos,\ypos)[\tokenc`\tokend`{#7}]{\width}{\arrowtyped}b%
\advance\ypos by \height
\puthmorphism(\xpos,\ypos)[\tokena`\tokenb`{#4}]{\width}{\arrowtypea}a%
\putvmorphism(\xpos,\ypos)[``{#5}]{\height}{\arrowtypeb}l%
\advance\xpos by \width
\putvmorphism(\xpos,\ypos)[``{#6}]{\height}{\arrowtypec}r%
}}
 
\def\putsquare{\@ifnextchar <{\putsquarep}{\putsquarep%
   <\arrowtypea`\arrowtypeb`\arrowtypec`\arrowtyped;\width`\height>}}
\def\square{\@ifnextchar< {\squarep}{\squarep
   <\arrowtypea`\arrowtypeb`\arrowtypec`\arrowtyped;\width`\height>}}
\def\squarep<#1>[#2`#3`#4`#5;#6`#7`#8`#9]{{
\setsqparms[#1]
\diagram
\putsquarep<\arrowtypea`\arrowtypeb`\arrowtypec`
\arrowtyped;\width`\height>
(0,0)[#2`#3`#4`{#5};#6`#7`#8`{#9}]
\enddiagram
}}                                                 
\def\putptrianglep<#1>(#2,#3)[#4`#5`#6;#7`#8`#9]{{%
\settriparms[#1]%
\xpos=#2 \ypos=#3
\advance\ypos by \height
\puthmorphism(\xpos,\ypos)[#4`#5`{#7}]{\height}{\arrowtypea}a%
\putvmorphism(\xpos,\ypos)[`#6`{#8}]{\height}{\arrowtypeb}l%
\advance\xpos by\height
\putmorphism(\xpos,\ypos)(-1,-1)[``{#9}]{\height}{\arrowtypec}r%
}}
 
\def\putptriangle{\@ifnextchar <{\putptrianglep}{\putptrianglep
   <\arrowtypea`\arrowtypeb`\arrowtypec;\height>}}
\def\ptriangle{\@ifnextchar <{\ptrianglep}{\ptrianglep
   <\arrowtypea`\arrowtypeb`\arrowtypec;\height>}}
\def\ptrianglep<#1>[#2`#3`#4;#5`#6`#7]{{
\settriparms[#1]
\diagram
\putptrianglep<\arrowtypea`\arrowtypeb`
\arrowtypec;\height>
(0,0)[#2`#3`#4;#5`#6`{#7}]
\enddiagram
}}                                            
 
\def\putqtrianglep<#1>(#2,#3)[#4`#5`#6;#7`#8`#9]{{%
\settriparms[#1]%
\xpos=#2 \ypos=#3
\advance\ypos by\height
\puthmorphism(\xpos,\ypos)[#4`#5`{#7}]{\height}{\arrowtypea}a%
\putmorphism(\xpos,\ypos)(1,-1)[``{#8}]{\height}{\arrowtypeb}l%
\advance\xpos by\height
\putvmorphism(\xpos,\ypos)[`#6`{#9}]{\height}{\arrowtypec}r%
}}
 
\def\putqtriangle{\@ifnextchar <{\putqtrianglep}{\putqtrianglep
   <\arrowtypea`\arrowtypeb`\arrowtypec;\height>}}
\def\qtriangle{\@ifnextchar <{\qtrianglep}{\qtrianglep
   <\arrowtypea`\arrowtypeb`\arrowtypec;\height>}}
\def\qtrianglep<#1>[#2`#3`#4;#5`#6`#7]{{
\settriparms[#1]
\width=\height                                
\diagram
\putqtrianglep<\arrowtypea`\arrowtypeb`
\arrowtypec;\height>
(0,0)[#2`#3`#4;#5`#6`{#7}]
\enddiagram
}}
 
\def\putdtrianglep<#1>(#2,#3)[#4`#5`#6;#7`#8`#9]{{%
\settriparms[#1]%
\xpos=#2 \ypos=#3
\puthmorphism(\xpos,\ypos)[#5`#6`{#9}]{\height}{\arrowtypec}b%
\advance\xpos by \height \advance\ypos by\height
\putmorphism(\xpos,\ypos)(-1,-1)[``{#7}]{\height}{\arrowtypea}l%
\putvmorphism(\xpos,\ypos)[#4``{#8}]{\height}{\arrowtypeb}r%
}}
 
\def\putdtriangle{\@ifnextchar <{\putdtrianglep}{\putdtrianglep
   <\arrowtypea`\arrowtypeb`\arrowtypec;\height>}}
\def\dtriangle{\@ifnextchar <{\dtrianglep}{\dtrianglep
   <\arrowtypea`\arrowtypeb`\arrowtypec;\height>}}
\def\dtrianglep<#1>[#2`#3`#4;#5`#6`#7]{{
\settriparms[#1]
\width=\height                                
\diagram
\putdtrianglep<\arrowtypea`\arrowtypeb`
\arrowtypec;\height>
(0,0)[#2`#3`#4;#5`#6`{#7}]
\enddiagram
}}
 
\def\putbtrianglep<#1>(#2,#3)[#4`#5`#6;#7`#8`#9]{{%
\settriparms[#1]%
\xpos=#2 \ypos=#3
\puthmorphism(\xpos,\ypos)[#5`#6`{#9}]{\height}{\arrowtypec}b%
\advance\ypos by\height
\putmorphism(\xpos,\ypos)(1,-1)[``{#8}]{\height}{\arrowtypeb}r%
\putvmorphism(\xpos,\ypos)[#4``{#7}]{\height}{\arrowtypea}l%
}}
 
\def\putbtriangle{\@ifnextchar <{\putbtrianglep}{\putbtrianglep
   <\arrowtypea`\arrowtypeb`\arrowtypec;\height>}}
\def\btriangle{\@ifnextchar <{\btrianglep}{\btrianglep
   <\arrowtypea`\arrowtypeb`\arrowtypec;\height>}}
\def\btrianglep<#1>[#2`#3`#4;#5`#6`#7]{{
\settriparms[#1]
\width=\height                               
\diagram
\putbtrianglep<\arrowtypea`\arrowtypeb`
\arrowtypec;\height>
(0,0)[#2`#3`#4;#5`#6`{#7}]
\enddiagram
}}
 
\def\putAtrianglep<#1>(#2,#3)[#4`#5`#6;#7`#8`#9]{{%
\settriparms[#1]%
\xpos=#2 \ypos=#3
{\multiply \height by2
\puthmorphism(\xpos,\ypos)[#5`#6`{#9}]{\height}{\arrowtypec}b}%
\advance\xpos by\height \advance\ypos by\height
\putmorphism(\xpos,\ypos)(-1,-1)[#4``{#7}]{\height}{\arrowtypea}l%
\putmorphism(\xpos,\ypos)(1,-1)[``{#8}]{\height}{\arrowtypeb}r%
}}
 
\def\putAtriangle{\@ifnextchar <{\putAtrianglep}{\putAtrianglep
   <\arrowtypea`\arrowtypeb`\arrowtypec;\height>}}
\def\Atriangle{\@ifnextchar <{\Atrianglep}{\Atrianglep
   <\arrowtypea`\arrowtypeb`\arrowtypec;\height>}}
\def\Atrianglep<#1>[#2`#3`#4;#5`#6`#7]{{
\settriparms[#1]
\width=\height                                     
\diagram
\putAtrianglep<\arrowtypea`\arrowtypeb`
\arrowtypec;\height>
(0,0)[#2`#3`#4;#5`#6`{#7}]
\enddiagram
}}
 
\def\putAtrianglepairp<#1>(#2)[#3;#4`#5`#6`#7`#8]{{%
\settripairparms[#1]%
\setpos(#2)%
\settokens`#3`%
\puthmorphism(\xpos,\ypos)[\tokenb`\tokenc`{#7}]{\height}{\arrowtyped}b%
\advance\xpos by\height
\puthmorphism(\xpos,\ypos)[\phantom{\tokenc}`\tokend`{#8}]%
{\height}{\arrowtypee}b%
\advance\ypos by\height
\putmorphism(\xpos,\ypos)(-1,-1)[\tokena``{#4}]{\height}{\arrowtypea}l%
\putvmorphism(\xpos,\ypos)[``{#5}]{\height}{\arrowtypeb}m%
\putmorphism(\xpos,\ypos)(1,-1)[``{#6}]{\height}{\arrowtypec}r%
}}
 
\def\putAtrianglepair{\@ifnextchar <{\putAtrianglepairp}{\putAtrianglepairp%
   <\arrowtypea`\arrowtypeb`\arrowtypec`\arrowtyped`\arrowtypee;\height>}}
\def\Atrianglepair{\@ifnextchar <{\Atrianglepairp}{\Atrianglepairp%
   <\arrowtypea`\arrowtypeb`\arrowtypec`\arrowtyped`\arrowtypee;\height>}}
 
\def\Atrianglepairp<#1>[#2;#3`#4`#5`#6`#7]{{
\settripairparms[#1]
\settokens`#2`
\width=\height                                
\diagram
\putAtrianglepairp                            
<\arrowtypea`\arrowtypeb`\arrowtypec`
\arrowtyped`\arrowtypee;\height>
(0,0)[{#2};#3`#4`#5`#6`{#7}]
\enddiagram
}}
 
\def\putVtrianglep<#1>(#2,#3)[#4`#5`#6;#7`#8`#9]{{%
\settriparms[#1]%
\xpos=#2 \ypos=#3
\advance\ypos by\height
{\multiply\height by2
\puthmorphism(\xpos,\ypos)[#4`#5`{#7}]{\height}{\arrowtypea}a}%
\putmorphism(\xpos,\ypos)(1,-1)[`#6`{#8}]{\height}{\arrowtypeb}l%
\advance\xpos by\height
\advance\xpos by\height
\putmorphism(\xpos,\ypos)(-1,-1)[``{#9}]{\height}{\arrowtypec}r%
}}
 
\def\putVtriangle{\@ifnextchar <{\putVtrianglep}{\putVtrianglep
   <\arrowtypea`\arrowtypeb`\arrowtypec;\height>}}
\def\Vtriangle{\@ifnextchar <{\Vtrianglep}{\Vtrianglep
   <\arrowtypea`\arrowtypeb`\arrowtypec;\height>}}
\def\Vtrianglep<#1>[#2`#3`#4;#5`#6`#7]{{
\settriparms[#1]
\width=\height                                 
\diagram
\putVtrianglep<\arrowtypea`\arrowtypeb`
\arrowtypec;\height>
(0,0)[#2`#3`#4;#5`#6`{#7}]
\enddiagram
}}
 
\def\putVtrianglepairp<#1>(#2)[#3;#4`#5`#6`#7`#8]{{
\settripairparms[#1]%
\setpos(#2)%
\settokens`#3`%
\advance\ypos by\height
\putmorphism(\xpos,\ypos)(1,-1)[`\tokend`{#6}]{\height}{\arrowtypec}l%
\puthmorphism(\xpos,\ypos)[\tokena`\tokenb`{#4}]{\height}{\arrowtypea}a%
\advance\xpos by\height
\puthmorphism(\xpos,\ypos)[\phantom{\tokenb}`\tokenc`{#5}]%
{\height}{\arrowtypeb}a%
\putvmorphism(\xpos,\ypos)[``{#7}]{\height}{\arrowtyped}m%
\advance\xpos by\height
\putmorphism(\xpos,\ypos)(-1,-1)[``{#8}]{\height}{\arrowtypee}r%
}}
 
\def\putVtrianglepair{\@ifnextchar <{\putVtrianglepairp}{\putVtrianglepairp%
    <\arrowtypea`\arrowtypeb`\arrowtypec`\arrowtyped`\arrowtypee;\height>}}
\def\Vtrianglepair{\@ifnextchar <{\Vtrianglepairp}{\Vtrianglepairp%
    <\arrowtypea`\arrowtypeb`\arrowtypec`\arrowtyped`\arrowtypee;\height>}}
\def\Vtrianglepairp<#1>[#2;#3`#4`#5`#6`#7]{{
\settripairparms[#1]
\settokens`#2`
\diagram
\putVtrianglepairp                             
<\arrowtypea`\arrowtypeb`\arrowtypec`
\arrowtyped`\arrowtypee;\height>
(0,0)[{#2};#3`#4`#5`#6`{#7}]
\enddiagram
}}

\def\putCtrianglep<#1>(#2,#3)[#4`#5`#6;#7`#8`#9]{{%
\settriparms[#1]%
\xpos=#2 \ypos=#3
\advance\ypos by\height
\putmorphism(\xpos,\ypos)(1,-1)[``{#9}]{\height}{\arrowtypec}l%
\advance\xpos by\height
\advance\ypos by\height
\putmorphism(\xpos,\ypos)(-1,-1)[#4`#5`{#7}]{\height}{\arrowtypea}l%
{\multiply\height by 2
\putvmorphism(\xpos,\ypos)[`#6`{#8}]{\height}{\arrowtypeb}r}%
}}
 
\def\putCtriangle{\@ifnextchar <{\putCtrianglep}{\putCtrianglep
    <\arrowtypea`\arrowtypeb`\arrowtypec;\height>}}
\def\Ctriangle{\@ifnextchar <{\Ctrianglep}{\Ctrianglep
    <\arrowtypea`\arrowtypeb`\arrowtypec;\height>}}
\def\Ctrianglep<#1>[#2`#3`#4;#5`#6`#7]{{
\settriparms[#1]
\width=\height                               
\diagram
\putCtrianglep<\arrowtypea`\arrowtypeb`
\arrowtypec;\height>
(0,0)[#2`#3`#4;#5`#6`{#7}]
\enddiagram
}}                                           
\def\putDtrianglep<#1>(#2,#3)[#4`#5`#6;#7`#8`#9]{{%
\settriparms[#1]%
\xpos=#2 \ypos=#3
\advance\xpos by\height \advance\ypos by\height
\putmorphism(\xpos,\ypos)(-1,-1)[``{#9}]{\height}{\arrowtypec}r%
\advance\xpos by-\height \advance\ypos by\height
\putmorphism(\xpos,\ypos)(1,-1)[`#5`{#8}]{\height}{\arrowtypeb}r%
{\multiply\height by 2
\putvmorphism(\xpos,\ypos)[#4`#6`{#7}]{\height}{\arrowtypea}l}%
}}
 
\def\putDtriangle{\@ifnextchar <{\putDtrianglep}{\putDtrianglep
    <\arrowtypea`\arrowtypeb`\arrowtypec;\height>}}
\def\Dtriangle{\@ifnextchar <{\Dtrianglep}{\Dtrianglep
   <\arrowtypea`\arrowtypeb`\arrowtypec;\height>}}
\def\Dtrianglep<#1>[#2`#3`#4;#5`#6`#7]{{
\settriparms[#1]
\width=\height                              
\diagram
\putDtrianglep<\arrowtypea`\arrowtypeb`
\arrowtypec;\height>
(0,0)[#2`#3`#4;#5`#6`{#7}]
\enddiagram
}}                                          
\def\setrecparms[#1`#2]{\width=#1 \height=#2}%
 
\def\recursep<#1`#2>[#3;#4`#5`#6`#7`#8]{{\m@th
\width=#1 \height=#2
\settokens`#3`
\settowidth{\tempdimen}{$\tokena$}
\ifdim\tempdimen=0pt
  \savebox{\tempboxa}{\hbox{$\tokenb$}}%
  \savebox{\tempboxb}{\hbox{$\tokend$}}%
  \savebox{\tempboxc}{\hbox{$#6$}}%
\else
  \savebox{\tempboxa}{\hbox{$\hbox{$\tokena$}\times\hbox{$\tokenb$}$}}%
  \savebox{\tempboxb}{\hbox{$\hbox{$\tokena$}\times\hbox{$\tokend$}$}}%
  \savebox{\tempboxc}{\hbox{$\hbox{$\tokena$}\times\hbox{$#6$}$}}%
\fi
\ypos=\height
\divide\ypos by 2
\xpos=\ypos
\advance\xpos by \width
\bfig
\putCtrianglep<-1`1`1;\ypos>(0,0)[`\tokenc`;#5`#6`{#7}]%
\puthmorphism(\ypos,0)[\tokend`\usebox{\tempboxb}`{#8}]{\width}{-1}b%
\puthmorphism(\ypos,\height)[\tokenb`\usebox{\tempboxa}`{#4}]{\width}{-1}a%
\advance\ypos by \width
\putvmorphism(\ypos,\height)[``\usebox{\tempboxc}]{\height}1r%
\efig
}}
 
\def\recurse{\@ifnextchar <{\recursep}{\recursep<\width`\height>}}
 
\def\puttwohmorphisms(#1,#2)[#3`#4;#5`#6]#7#8#9{{%
\puthmorphism(#1,#2)[#3`#4`]{#7}0a
\ypos=#2
\advance\ypos by 20
\puthmorphism(#1,\ypos)[\phantom{#3}`\phantom{#4}`#5]{#7}{#8}a
\advance\ypos by -40
\puthmorphism(#1,\ypos)[\phantom{#3}`\phantom{#4}`#6]{#7}{#9}b
}}
 
\def\puttwovmorphisms(#1,#2)[#3`#4;#5`#6]#7#8#9{{%
\putvmorphism(#1,#2)[#3`#4`]{#7}0a
\xpos=#1
\advance\xpos by -20
\putvmorphism(\xpos,#2)[\phantom{#3}`\phantom{#4}`#5]{#7}{#8}l
\advance\xpos by 40
\putvmorphism(\xpos,#2)[\phantom{#3}`\phantom{#4}`#6]{#7}{#9}r
}}
 
\def\puthcoequalizer(#1)[#2`#3`#4;#5`#6`#7]#8#9{{%
\setpos(#1)%
\puttwohmorphisms(\xpos,\ypos)[#2`#3;#5`#6]{#8}11%
\advance\xpos by #8
\puthmorphism(\xpos,\ypos)[\phantom{#3}`#4`#7]{#8}1{#9}
}}
 
\def\putvcoequalizer(#1)[#2`#3`#4;#5`#6`#7]#8#9{{%
\setpos(#1)%
\puttwovmorphisms(\xpos,\ypos)[#2`#3;#5`#6]{#8}11%
\advance\ypos by -#8
\putvmorphism(\xpos,\ypos)[\phantom{#3}`#4`#7]{#8}1{#9}
}}
 
\def\putthreehmorphisms(#1)[#2`#3;#4`#5`#6]#7(#8)#9{{%
\setpos(#1) \settypes(#8)
\if a#9 %
     \vertsize{\tempcounta}{#5}%
     \vertsize{\tempcountb}{#6}%
     \ifnum \tempcounta<\tempcountb \tempcounta=\tempcountb \fi
\else
     \vertsize{\tempcounta}{#4}%
     \vertsize{\tempcountb}{#5}%
     \ifnum \tempcounta<\tempcountb \tempcounta=\tempcountb \fi
\fi
\advance \tempcounta by 60
\puthmorphism(\xpos,\ypos)[#2`#3`#5]{#7}{\arrowtypeb}{#9}
\advance\ypos by \tempcounta
\puthmorphism(\xpos,\ypos)[\phantom{#2}`\phantom{#3}`#4]{#7}{\arrowtypea}{#9}
\advance\ypos by -\tempcounta \advance\ypos by -\tempcounta
\puthmorphism(\xpos,\ypos)[\phantom{#2}`\phantom{#3}`#6]{#7}{\arrowtypec}{#9}
}}
 
\def\setarrowtoks[#1`#2`#3`#4`#5`#6]{%
\def\toka{#1}
\def\tokb{#2}
\def\tokc{#3}
\def\tokd{#4}
\def\toke{#5}
\def\tokf{#6}
}
\def\hex{\@ifnextchar <{\hexp}{\hexp<1000`400>}}
\def\hexp<#1`#2>[#3`#4`#5`#6`#7`#8;#9]{%
\setarrowtoks[#9]
\yext=#2 \advance \yext by #2
\xext=#1 \advance\xext by \yext
\bfig
\putCtriangle<-1`0`1;#2>(0,0)[`#5`;\tokb``\tokd]
\xext=#1 \yext=#2 \advance \yext by #2
\putsquare<1`0`0`1;\xext`\yext>(#2,0)[#3`#4`#7`#8;\toka```\tokf]
\advance \xext by #2
\putDtriangle<0`1`-1;#2>(\xext,0)[`#6`;`\tokc`\toke]
\efig
}
\makeatother